\documentclass[11pt]{amsart}
\usepackage[margin=3cm]{geometry}
\title{Martingales and descent statistics}
\author{Alperen  \"{O}zdemir}
\address{Department of Mathematics, Georgia Institute of Technology} 
\email{aozdemir6@gatech.edu} 

\usepackage{amssymb,amsmath, amsthm, mathtools}
\usepackage[colorlinks=true, urlcolor=black,linkcolor=blue, citecolor=blue]{hyperref}

\newtheorem{theorem}{Theorem}[section]
\newtheorem{lemma}{Lemma}[section]

\newtheorem{proposition}{Proposition}[section]

\DeclareRobustCommand{\Stirling}{\genfrac\{\}{0pt}{}}

\newcommand{\bbox}{\hfill $\Box$}
\newcommand{\pf}{\noindent {\it Proof:} }


\begin{document}

\begin{abstract}
This paper develops techniques to study the number of descents in random permutations via martingales. We relax an assumption in the Berry-Esseen theorem of Bolthausen (1982) to extend the theorem's scope to martingale differences of time-dependent variances. This extension leads to a new proof of the fact that the number of descents in random permutations is asymptotically normal with an error bound of order $1/\sqrt{n}.$  The same techniques are shown to be applicable to other descent and descent-related statistics as they satisfy certain recurrence relation conditions. These statistics include inversions, descents in signed permutations, descents in Stirling permutations, the length of the longest alternating subsequences, descents in matchings and two-sided Eulerian numbers. 
\end{abstract}

\maketitle

\section{Introduction} \label{intro}
Let $S_n$ be the symmetric group defined over $n$ elements. A permutation $\pi \in S_n$ is said to have a descent at position $i$ if $\pi(i) > \pi(i+1).$  Let $D_n$ be the random variable counting the number of descents of a random permutation in $S_n.$ 

The procedure which leads to a martingale is as follows. Given a permutation $\pi$ in one-line notation (i.e., $\pi=\pi(1) \pi(2) \cdots \pi(n)$), inserting $(n+1)$ in a random position yields a descent except the terminal position. Yet if it is inserted right next to a position where $\pi$ has a descent, the descent is broken up. Therefore, after inserting $(n+1),$ there are exactly $D_n+1$ positions that the number of descents stays the same, otherwise it increases by 1. Then for the sigma field $\mathcal{F}_n= \sigma(D_1,\cdots, D_n),$ we have
\begin{align*}
\mathbf{E}(D_{n+1}| \mathcal{F}_n) &= D_n \frac{D_n+1}{n+1} + (D_n+1) \frac{n-D_n}{n+1} \\
&= \frac{n}{n+1} D_n + \frac{n}{n+1}. 
\end{align*}
The mean of $D_n$ is $\frac{n-1}{2}$ by symmetry. Subtracting the mean followed by a proper scaling gives the martingale below.
\begin{equation*}
Z_n :=n \left( D_n - \frac{n-1}{2} \right).
\end{equation*}
Observe that $\mathbf{E}(Z_n)=0$ and $\mathbf{E}(Z_{n}|\mathcal{F}_m)=Z_m$ for $n>m.$ Therefore, $\{Z_n, \mathcal{F}_n \}$ is a zero-mean martingale for $n\geq 1.$ We want to show 
\begin{theorem} \label{main}
Let $D_n$ be the number of descents in a uniformly chosen permutation from $S_n.$ Then
\begin{equation*}
\sup_{x \in \mathbb{R}} \left| \mathbf{P} \left(\frac{D_n - \frac{n-1}{2}}{\sqrt{\frac{n+1}{12}}} \leq x\right) - \Phi(x) \right| \leq \frac{C}{\sqrt{n}},
\end{equation*}
where $\Phi$ is the standard normal distribution and $C$ is a constant independent of $n$.
\end{theorem}
The asymptotic normality, not necessarily the order of convergence, in the theorem has been proved by various different methods. We recall that the number of permutations in $S_n$ with  $k$ descents is the Eulerian number $A_{n,k}$. The fact that the Eulerian polynomial $A_n(t)= \sum_{k} A_{n,k} t^k$ has only real roots, which was known to Frobenius \cite{Fr}, gives asymptotic normality by Harper's approach \cite {Ha}. A proof by the method of moments is given in \cite{DB}. Bender \cite{B} expanded the generating function for Eulerian numbers around its singularities to approximate the characteristic function of the distribution of number of descents. Indeed, the explicit formula for $A(n,k)$ matches with the distribution function of $n$ uniform random variables whose sum is between $k$ and $(k+1),$ which was shown by Tanny \cite{T}. Fulman \cite{F} used Stein's method of exchangable pairs to prove the theorem. Also central limit theorems involving locally dependent random variables can be employed to show the result, see Corollary 2 of \cite{BR89}. In this paper, we use martingale limit theorems to prove the asymptotic normality and the order of convergence, then apply it to other combinatorial statistics. The paper consists of two parts. 

In the first part, we study the moments of $D_n$ in Section \ref{Mom}, and then show the asymptotic normality by a classical result in Section \ref{clt}. Section \ref{Bolt} is about the rate of convergence in the central limit theorem, and contains the proof of Theorem \ref{main}. We argue that a central limit theorem in \cite{Bh} that gives an order of convergence of $1/\sqrt{n}$ can be applied to our case after relaxing one of its assumptions. The assumption is that the variance of the martingale differences is asymptotically constant, which holds true for stationary martingale differences but leaves out non-stationary examples of increasing variance as in our case. We give a proof at the end of the paper, Section \ref{Pf}, that the polynomial growth of the variance of the martingale differences is sufficient to obtain an error bound of order $1/\sqrt{n}$. 

In the second part, we abstract the martingale formulation from the initial idea by means of recurrence relations and apply it to other descent-related statistics. In Section \ref{Recform}, we show how to obtain martingales from a given set of recurrence relations. Then we give a formula of certain combinatorial interest, which gives the recurrence relation coefficients for a class of permutation statistics. Section \ref{App} supplies examples. 

\section{Moments} \label{Mom}
We already noted that $\mathbf{E}(D_n)=\frac{n-1}{2}.$ The variance of $D_n$ is also well-known to be $\frac{n+1}{12}.$ The calculation of the variance can be found in Chapter 6 of \cite{Bo}; we adopt the same technique therein to obtain the fourth central moment. One can also use the method of moments as in \cite{CK}. 

\indent First define $2$-dependent random variables over $S_n$ as follows:
\begin{equation*}
   T_i(\pi) = 
    \begin{cases}
      1  \text{ if } \pi(i) > \pi(i+1),\\
      0 \text{ otherwise.} 
    \end{cases}
\end{equation*}
Clearly, $D_n=\sum_{i=1}^{n-1} T_i.$ The third moment can be written as
\begin{equation*}
\mathbf{E}(D_n^3)=\sum_{i} \mathbf{E}(T_i^3) + 6 \sum_{i < j} \mathbf{E}(T_i^2 T_j) +  6 \sum_{i < j < k} \mathbf{E}(T_i T_j T_k).  
\end{equation*}
The first term satisfies $\sum_{i} \mathbf{E}(T_i^3)=\sum_{i} \mathbf{E}(T_i)=\frac{n-1}{2}.$ For the second term, $6 \sum_{i < j} \mathbf{E}(T_i^2 T_j)=6 \sum_{i < j} \mathbf{E}(T_i T_j),$ we use $2$-dependence of $T_i'$s. There are $(n-2)$ cases where $j=i+1,$ for which $\mathbf{E}(T_i T_j)=\frac{1}{6}.$ For the remaining $\binom{n-2}{2}$ cases, $i$ and $j$ differ by more than $1,$ so they are independent. The expected value in the latter case is $\frac{1}{4}.$ Analyzing the cases in a similar fashion for the last term yields
\begin{align*}
\mathbf{E}(D_n^3) &= \frac{n-1}{2} + 6 \left(\frac{n-2}{6}+  \frac{\binom{n-2}{2}}{4}\right) + 6 \left(\frac{n-3}{24}+ 2 \frac{\binom{n-3}{2}}{12}+ \frac{\binom{n-3}{3}}{8} \right) \notag \\
&=\frac{(n^2-n+2)(n-1)}{8}.
\end{align*}
The fourth moment is given by
\begin{align*}
\mathbf{E}(D_n^4) =&\sum_{i} \mathbf{E}(T_i^4) + 8 \sum_{i < j} \mathbf{E}(T_i^3 T_j) + 6 \sum_{i < j} \mathbf{E}(T_i^2 T_j^2) + \\ & 36 \sum_{i < j < k} \mathbf{E}(T_i^2 T_j T_k)+ 24 \sum_{i < j < k<l} \mathbf{E}(T_i T_j T_k T_l).  
\end{align*} 
By the same argument as above,
\begin{equation*}
\mathbf{E}(D_n^4)=\frac{n^4}{16}-\frac{n^3}{8} + \frac{13n^2}{48}+\mathcal{O}(n).
\end{equation*}
Finally, we calculate the fourth central moment, 
\begin{equation} \label{fourth}
\mathbf{E}[(D_n-\mathbf{E}(D_n))^4]=\frac{n^2}{48}+\mathcal{O}(n).
\end{equation}

\section{Asymptotic normality} \label{clt}

A comprehensive account of martingale limit theorems and related topics is \cite{HH}. The central limit theorem therein requires conditions analogous to those of classical central limit theorems for sums of independent random variables. The first one is the Lindeberg condition. The second condition is about the convergence in probability of the conditional variance of martingale differences. The limiting behavior of a martingale is essentially determined by this quantity or its counterpart $U^2_n := \sum X_i^2$ where $X_i$ is the $i$th martingale difference. See Chapter 2 of \cite{HH} for their relationship. 

A triangular array with small deviations for martingale differences can be obtained as prescribed in \cite{HH}. Take $\mathcal{F}_{n,i}=\mathcal{F}_i$ and define $Z_{n,i}= s_n^{-1} Z_i$ for $1 \leq i \leq n$ where $s_n$ is the standard deviation of $Z_n$. In this way, the expectation of the conditional variance is normalized to be $1$. 

The central limit theorem reads:
\begin{theorem}\textnormal{(\cite{HH})} \label{thm2}
Let $\{Z_{n,i}, \mathcal{F}_{n,i}, 1 \leq i \leq k_n, n\geq 1\}$ be a zero-mean, square integrable martingale array with differences $X_{n,i}:= Z_{n,i}-Z_{n,i\text{-}1},$ and $
\eta^2$ be an almost surely finite random variable. Suppose that 
\begin{gather}
\text{ for all } \epsilon>0, \quad \sum_{i=1}^{k_n}\mathbf{E}[X_{n,i}^2 I(|X_{n,i}|>\epsilon)|\mathcal{F}_{n,i\text{-}1}] \overset{p}{\to} 0, \label{cona1}\\
\sum_{i=1}^{k_n} \mathbf{E}[X_{n,i}^2 | \mathcal{F}_{n, i\text{-}1}] \overset{p}{\to} \eta^2 \label{cona2}  
\end{gather}
and the $\sigma$-fields are nested, i.e., $\mathcal{F}_{n,i\text{-}1} \subseteq \mathcal{F}_{n,i}$ for $1 \leq i \leq k_n, n\geq 1. $\\
Then $Z_{n, k_n} = \sum_i X_{n,i} \overset{d}{\to} Z,$ where $Z$ has characteristic function $\mathbf{E} [\exp\left(-\frac{1}{2} \eta^2 t^2 \right)].$
\end{theorem}

\begin{proof}[Proof of the asymptotic normality in Theorem \ref{main}]
We verify the conditions of Theorem \ref{thm2}. Take the row length $k_n$ of the array in the theorem to be equal to $n$. Notice that $s_n=n\sqrt{\frac{n+1}{12}}$ in our case.  Looking at the martingale differences, $X_{n,i}:= Z_{n,i}-Z_{n,i\text{-}1},$ we have
\begin{equation}\label{defe}
X_{n,i}= \frac{\sqrt{12}}{n\sqrt{n+1}}  \left( iD_i -(i-1) D_{i-1} -(i-1)  \right).
\end{equation}
Let $W_i:=D_i-\frac{i-1}{2}$ be the zero-mean random variable for the number of descents. Then
\begin{equation} \label{wi}
\begin{aligned}
  X_{n,i}|\mathcal{F}_{n,i \text{-}1}= &   \frac{\sqrt{12}}{n\sqrt{n+1}} 
   \begin{cases}
     W_{i-1}-\frac{i}{2},&  \text{with prob. } \frac{1}{2}+\frac{W_{i \text{-} 1}}{i},\\
       W_{i-1}+\frac{i}{2},& \text{with prob. } \frac{1}{2}-\frac{W_{i-1}}{i}.
    \end{cases}
\end{aligned}
\end{equation}

We first verify the Lindeberg condition (\ref{cona1}). By \eqref{defe} and the fact that $D_i\leq i-1,$ we have $|X_{n,i}|\leq \frac{\sqrt{12} \, i}{n\sqrt{n+1}}.$ So, for given $\epsilon>0$ and sufficiently large $n$, $P(|X_{n,i}|>\epsilon)=0$ where $1 \leq i \leq n-1.$ Hence, (\ref{cona1}) is satisfied. 

Next we verify the convergence of conditional variance (\ref{cona2}). By (\ref{wi}),
\begin{equation*}
\mathbf{E}[X_{n,i}^2 | \mathcal{F}_{n,i\text{-}1}] =  \frac{12}{n^2 (n+1)^2}\left( \frac{i^2}{4} - W_{i \text{-} 1}^2 \right).
\end{equation*}
Then the conditional variance for $Z_{nn}$ is 
 \begin{equation*} 
\begin{aligned}
V_{n,n}^2:=\sum_i \mathbf{E}[X_{n,i}^2 | \mathcal{F}_{n,i\text{-}1}] &=  \frac{12}{n^2 (n+1)^2} \left(\sum_i \frac{i^2}{4} - \sum_i W_{i \text{-} 1}^2 \right), \\
&= 1+ \frac{1}{2n} - \frac{12}{n^2(n+1)} \sum_i W_{i \text{-} 1}^2. 
\end{aligned}
\end{equation*}
Observe that
\begin{align*}
\mathbf{E}(V_{n,n}^2) &=1+ \frac{1}{2n}-\frac{12}{n^2(n+1)} \sum_i \mathbf{E}(W_{i \text{-} 1}^2) \\
&= 1+ \frac{1}{2n}-\frac{12}{n^2(n+1)} \sum_i \frac{i}{12} \\
&=1. 
\end{align*}
Below, we use Cauchy-Schwarz inequality and \eqref{fourth} to show that the variance of $V_{nn}^2$ converges to $0$. 
\begin{align*}
\mathbf{E}\left[\left(V_{n,n}^2-1-\frac{1}{2n}\right)^2\right] &=\frac{144}{n^4 (n+1)^2} \left( \sum_{1 \leq i,j\leq n\text{-}1} \mathbf{E}\left(W_{i \text{-} 1}^2 W_{j \text{-} 1}^2\right)  \right)\\
& \leq \frac{144}{n^4 (n+1)^2} \left( \sum_{1 \leq i,j\leq n\text{-}1} \sqrt{\mathbf{E}(W_{i \text{-} 1}^4) \mathbf{E}( W_{j \text{-} 1}^4)}  \right)\\
 & \leq  \frac{C}{n^4 (n+1)^2} \left(\sum_{1 \leq i,j\leq n\text{-}1}  i j \right)\\
& \leq \frac{C  }{ (n+1)^2} \rightarrow 0.
\end{align*}
Finally, Chebyschev's inequality gives that $V_{n,n}^2=\sum_i \mathbf{E}[X_{n,i}^2 | \mathcal{F}_{n,i\text{-}1}]$ converges to $1$ in probability. Both conditions are verified; the result follows from Theorem \ref{thm2}.
\end{proof}

\section{Rate of Convergence} \label{Bolt}

The order of convergence in the central limit theorem for sums of independent random variables is known to be $n^{-1/2}$ under mild conditions and the uniform boundedness assumption of the third moments; see Chapter 5 of \cite{Pet}. While the same assumption gives a rate of order $n^{-1/4}$ at best in the case of martingale differences due to the dependency \cite{HH}. Bolthausen \cite{Bh} developed a technique, which is based on Bergstr\"om's inductive method \cite{Ber}, that yields rates comparable to the former case by imposing the uniform boundedness of the fourth conditional moment. A restriction compared to the independency case is that the variances of martingale differences are assumed to be asymptotically constant, i.e., $\lim_n \mathbf{E}[X_n^2] = \sigma^2 > 0.$ This assumption turns out to be restrictive in some combinatorial applications. For instance, in the case of the number of descents, the above quantity grows at a rate of $n^2.$ Another example is the martingale differences for the character ratios of the symmetric group where the variance grows linearly in $n$ \cite{F2}, and several other examples are given in Section \ref{App}. 

Various refinements in the result mentioned above have been obtained by modifications of the Bolthausen's technique (see, e.g., \cite{Hae, Re, RR}), mostly regarding the assumption on the fourth conditional moment. To the best of authors's knowledge, the time dependence of the variance has not been addressed. We show that the assumption can be relaxed to allow the variance of martingale differences to grow at a polynomial rate while the same rate of convergence, $n^{-1/2},$ is maintained. 

Let $X_i=Z_i-Z_{i-1}$ and define $\sigma_i^2= \mathbf{E}[X_i^2].$  Observe that $s_n^2= \mathbf{E}[Z_n^2]=\sum_{i=1}^n \sigma_i^2$ which follows from the fact that $\mathbf{E}[X_i X_j]=0$ for $i\neq j$ as $X_i$ and $X_j$ are martingale differences. The condition
\begin{equation}\label{maincond}
1 \leq \liminf_n \sqrt{n}\frac{\sigma_{n+1}}{s_n}\leq\limsup_n \sqrt{n}\frac{\sigma_{n+1}}{s_n} < \infty
\end{equation}
replaces Assumption (1.2) in \cite{Bh}, which can equivalently be stated as $\lim_{n} \sqrt{n}\sigma_{n+1} / s_n =1$. The theorem is as follows:
\begin{theorem}\label{convthm}
Let $\{Z_{n}, \mathcal{F}_{n}, n\geq 1\}$ be a zero-mean martingale with martingale differences $X_{i}$ that satisfy \eqref{maincond} and  $Y_i=X_i/\sigma_i.$ If 
\begin{gather}
\sup_i \|\mathbf{E}(Y_{i}^4 | \mathcal{F}_{i-1}) \|_\infty < \infty, \label{bound4}\\
\sup_i \sqrt[2p']{i} \, \|\mathbf{E}(Y_{i}^3 | \mathcal{F}_{i-1}) \|_{p'} < \infty , \label{bound3}\\
 \sup_i \sqrt{i}\,  \|\mathbf{E}(Y_{i}^2 | \mathcal{F}_{i-1})-1 \|_{p} < \infty  \label{bound2} 
\end{gather}

for some $p,p' > 1,$ then 

\begin{equation*}
\sup_{t \in \mathbb{R}} |P(Z_n/s_n \leq t) - \Phi(t) | \leq \frac{C}{\sqrt{n}}
\end{equation*}
where $C$ is a constant independent of $n$ and $\Phi$ is the standard normal distribution.
\end{theorem}
The proof of Theorem \ref{convthm} is an adaptation of Bolthausen's technique. Its starting point is the Lindeberg's swapping argument. The key component of the proof is the recursive bound for the Kolmogorov distance between the martingle and the normal distribution. Since the proof involves technical details, such as term-by-term estimates in the Taylor approximation, we postpone it to Section \ref{Pf}. Here we use the result to prove the main theorem stated in the Introduction.
 
\begin{proof}[Proof of Theorem \ref{main}] We do not have recourse to the triangular arrays as in the previous section, so we do not have the additional scaling factor for the martingale differences, i.e.,
\begin{equation} \label{differance}
X_i | \mathcal{F}_{i-1} =\begin{aligned}
   &  
   \begin{cases}
     W_{i-1}-\frac{i}{2},&  \text{with prob. } \frac{1}{2}+\frac{W_{i \text{-} 1}}{i},\\
       W_{i-1}+\frac{i}{2},& \text{with prob. } \frac{1}{2}-\frac{W_{i-1}}{i}.
    \end{cases}
\end{aligned}
\end{equation}
We then have
\begin{align*}
\sigma_i^2 &= \mathbf{E}[\mathbf{E}[X_{i}^2| \mathcal{F}_{i-1}]]= \mathbf{E}\left[\frac{i^2}{4} - W_{i-1}^2\right]= \frac{3i^2 -i}{12},\\
s_n^2&=\sum_{i=1}^n \sigma_i^2 = \frac{n^2 (n+1)}{12}.
\end{align*}
So the variance of martingale differences is of polynomial order. In fact $\lim_{n} \sqrt{n}\sigma_{n+1} / s_n=\sqrt{3};$ thus, the condition \eqref{maincond} is satisfied. Next, we verify the other assumptions of the Theorem \ref{convthm}. 
Regarding the fourth conditional moment,
\begin{align*}
\|\mathbf{E}[X_i^4|\mathcal{F}_{i-1}]\|_{\infty}&=\left\|\frac{i^4}{16}+\frac{i^2 W_{i-1}^2}{2} - 3 W_{i-1}^4\right\|_{\infty} = \mathcal{O}(i^4)
\end{align*}
as $|W_{i-1}|\leq \frac{i-2}{2}.$ Then since $\sigma_i$ is of order $i,$ $\|\mathbf{E}[Y_i^4|\mathcal{F}_{i-1}]\|_{\infty}$ is uniformly bounded. Hence, \eqref{bound4} is satisfied. For the next condition, we bound 
\begin{equation} \label{norm3}
\|\mathbf{E}[X_i^3|\mathcal{F}_{i-1}]\|_{p}^p= \mathbf{E}\left[\left|\frac{i^2 W_{i-1}}{2}-2 W_{i-1}^3 \right|^{p}\right].
\end{equation}
We take $p=4/3$ in $\eqref{norm3},$ then it is bounded by
\begin{align*}
& \mathbf{E}\left[\left(\sqrt[3]{\frac{i^2 |W_{i-1}|}{2}} +\sqrt[3]{2} \, |W_{i-1}| \right)^{4}\right] \\
 \leq &C \, \mathbf{E}(i^{8/3}|W_{i-1}|^{4/3}+i^{2}|W_{i-1}|^{2}+i^{4/3}|W_{i-1}|^{8/3}+i^{2/3}|W_{i-1}|^{10/3}+|W_{i-1}|^{4})\\
= &\mathcal{O}(i^{10/3}),
\end{align*}
where the last line follows from Lyapunov's inequality,
\begin{equation*}
\sqrt[r]{\mathbf{E}|W_{i-1}|^r} \leq \sqrt[4]{\mathbf{E}|W_{i-1}|^4} = \mathcal{O}(\sqrt{i})
\end{equation*}
for $0 <r <4.$ Therefore, $\sqrt[8/3]{i}\|\mathbf{E}[Y_i^3|\mathcal{F}_{i-1}]\|_{4/3}$ is also uniformly bounded. For the last condition, we take $p'=2,$ so we have
\begin{align*}
\|\mathbf{E}[X_i^2|\mathcal{F}_{i-1}]-\sigma_i^2\|_{2}^2 &= \mathbf{E}\left[\left|\frac{i}{12} - W_{i-1}^2\right|^{2}\right] \\
& \leq \mathbf{E}\left[\frac{i^2}{144} + \frac{i W_{i-1}^2}{6} + W_{i-1}^4 \right] \\
& = \mathcal{O}(i^2).
\end{align*}
Taking the square root of the above expression and dividing by $\sigma_i^2,$ we can show that it is of order less than $\sqrt{i},$ which verifies \eqref{bound2}. Hence, by Theorem \ref{convthm} we obtain the asymptotic normality with an error term of order less than or equal to  $n^{-1/2}.$
\end{proof}
Before we move on to martingale formulation of other combinatorial statistics, we briefly discuss another application of Theorem \ref{convthm} without going into details. As mentioned above, it is shown in \cite{F2} that the normalized characters of the symmetric group evaluated at transpositions form a martingale sequence with respect to the Plancherel measure on the conjugacy classes of the symmetric group. The asymptotic normality is shown therein, and the convergence rate of order $n^{-1/2}$ is obtained in \cite{F3} and also in \cite{SQ} by variations of Stein's method, which reads as
\begin{equation*}
\sup_{x\in \mathbb{R}} \left|\mathbf{P}\left( \frac{n-1}{2} \frac{\chi^{\lambda}(12)}{\dim(\lambda)}\leq x\right) - \Phi(x) \right| \leq \frac{C}{\sqrt{n}},
\end{equation*}
where $\lambda$ is a conjugacy class chosen with respect to Plancherel measure. The conditional moments up to fourth degree are obtained in \cite{F2}. By taking $p=p'=2$ in Theorem \ref{convthm} and using Lemma 2.3. in \cite{F3}, we can give another proof of the above result.  

\section{Recurrence relations and martingale formulation} \label{Recform}

The primary goal of this section is to extend the martingale techniques to different permutation statistics. In the Introduction, we obtained the martingale for the number of descents in random permutation by inserting the new entry at a random position. However, as we will see in the next section, there are examples that this idea is not readily applicable or even not possible at all. Before introducing those examples in the next section, we show how to derive martingales from a broader class of recurrences. 

\subsection{Recurrence relations to martingales}
 Consider the Eulerian polynomial
\begin{equation*}
A_n(t)=\sum_{\pi \in S_n}t^{\textnormal{des}(\pi)}=\sum_{k\geq 0} A_{n,k} t^k,
\end{equation*}
where $\textnormal{des}(\pi)$ is the number of descents in $\pi \in S_n.$ Another definition of the Eulerian polynomial is the generating function
\begin{equation*}
\frac{A_n(t)}{(1-t)^{n+1}}= \sum_{k \geq 0} (k+1)^n t^k.
\end{equation*}
The coefficients in $A_n(t)$ can be expressed recursively by simple manipulations of the above sum. Foata's survey \cite{Fo} gives eight different expressions for Eulerian numbers. Consider
\begin{equation*}
A_{n+1,k} = (k+1) A_{n,k} + (n-k+1) A_{n, k-1}.
\end{equation*}
Dividing both sides by $A_{n+1}(1)=(n+1)!,$
\begin{align*}
\frac{A_{n+1,k}}{(n+1)!} = \frac{k+1}{n+1} \frac{A_{n,k}}{n!} + \frac{n-k+1}{n+1} \frac{A_{n,k-1}}{n!}.
\end{align*}
Equivalently,
\begin{equation*}
\mathbf{P}(D_{n+1}=k) = \frac{k+1}{n+1} \mathbf{P}(D_{n}=k) +  \frac{n-k+1}{n+1} \mathbf{P}(D_{n}=k-1).
\end{equation*}
In order to relate it to martingales, we write the next equation. 
\begin{equation*}
\mathbf{P}(D_{n+1}=k+1) = \frac{k+2}{n+1} \mathbf{P}(D_{n}=k+1) +  \frac{n-k}{n+1} \mathbf{P}(D_{n}=k).
\end{equation*}
Observe that the coefficients of $\mathbf{P}(D_{n}=k)$ in subsequent recurrences above agree with the probabilities of the martingale defined in the Introduction.

We next find conditions for martingale representation in a general setting. Let $P_{n,k}$ be nonnegative integers for $n=1,2,\dots$ and $0 \leq k \leq n,$ which is referred as \textit{combinatorial array} in \cite{Pit}. Also define $P_n(t)=\sum_{k=0}^n P_{n,k}t^k$ as before. Suppose $\{P_{n,k}\}$ satisfies recurrences of the form
\begin{align*}
P_{n+1,k}&=\alpha_{n+1,k}^{(0)} P_{n,k}&+& \quad \alpha_{n+1,k}^{(1)} P_{n,k-1} &+\cdots+& \quad\alpha_{n+1,k}^{(s)} P_{n,k-s}, \\
P_{n+1,k+1}&=\alpha_{n+1,k+1}^{(0)} P_{n,k+1}&+& \quad \alpha_{n+1,k+1}^{(1)} P_{n,k} &+\cdots +& \quad\alpha_{n+1,k+1}^{(s)} P_{n,k-s+1}, \\
& \vdotswithin{=} \\
P_{n+1,k+s}&=\alpha_{n+1,k+s}^{(0)} P_{n,k+s}&+& \quad \alpha_{n+1,k+s}^{(1)} P_{n,k+s-1} &+\cdots +& \quad\alpha_{n+1,k+s}^{(s)} P_{n,k}, 
\end{align*}
where $\alpha^{(i)}_{n+1,k} \geq 0$ for all $0 \leq i \leq s.$ In addition, we want the diagonal coefficients to add up to $\frac{P_{n+1}(1)}{P_n(1)}$ for all $k$ i.e.,
\begin{equation} \label{diag}
\alpha^{(0)}_{n+1,k}+ \alpha^{(1)}_{n+1,k+1}+ \cdots + \alpha^{(s)}_{n+1,k+s} =\frac{P_{n+1}(1)}{P_n(1)}.
\end{equation}
Let $p_{n+1,i}:= \alpha^{(i)}_{n+1,k+i} \cdot \frac{1}{P_{n+1}(1)/P_n(1)}$ to ease notation. Define the following random process,
\begin{equation} \label{martin}
Z_{n+1}=\begin{aligned}
   &   
   \begin{cases}
     Z_{n},&  \text{with prob. } p_{n+1,0},\\
       Z_{n}+1,& \text{with prob. } p_{n+1,1}, \\
       \quad \quad \, \vdots \\
       Z_{n}+s,& \text{with prob. } p_{n+1,s}. \\
    \end{cases}
\end{aligned}
\end{equation}
We observe that \eqref{diag} and nonnegativity of the coefficients are sufficient to conclude that $Z_n$ is a $submartingale$ with respect to $\mathcal{F}_n=\sigma(Z_1,\cdots,Z_n).$  A submartingale is defined in the same way with a martingale, but the difference is that we replace $``\mathbf{E}(Z_{n}|\mathcal{F}_m)= Z_m$'' by the weaker condition: \,$``\mathbf{E}(Z_{n}|\mathcal{F}_m) \geq Z_m$'' for $n>m.$ 
 A general formula on how to obtain a martingale from $Z_n$ is not possible with no further information on the probabilities involved in the definition of $Z_n.$ Nevertheless, in Section \ref{App}, we give examples of descent-related statistics for which we derive martingales from the submartingale obtained in the way described above. A  counterexample to \eqref{diag} is Stirling numbers of second kind, which is recursively defined as
\begin{equation} \label{stir}
\Stirling{n+1}{k}=k \Stirling{n}{k}+ \Stirling{n}{k-1}.
\end{equation}
So the method above does not apply in this case. 

\subsection{Rational generating functions to recurrence relations}
The next goal of this section is to derive recurrence relations from the rational generating functions of the form 
\begin{equation}\label{rat}
\frac{P_{dn}(t)}{(1-t)^{dn+1}}=\sum_{k\geq 0} f_n(k) t^k, 
\end{equation}
and use it in the next section. We note that $f_n$ has degree $d(n+1)$ if and only if $P(1) \neq 0$ (see Corollary 4.3.1 of \cite{EC1}). These functions appear in many different contexts. The order polynomials of labelled posets has rational generating functions, \cite{EC1} where the classical Eulerian polynomial is associated with the special case of an anti-chain. Another example is Ehrhart polynomials of convex polytopes; if the polytope is taken to be the unit cube, the generating function has Eulerian polynomial in the numerator (see Chapter 4 of \cite{EC1} and Chapter 8 of \cite{PEN}).

Suppose $f_{n+1}(k)=g_n(k)f_n(k)$ for some degree $d$ polynomial $g_n(k)$ for all $n.$ For instance, $d=1$ and $g_n(k)=k+1$ for Eulerian numbers. In general, if $g_n(k)=\alpha_nk+\beta_n$ for some $\alpha_n, \beta_n \in \mathbb{R}$, we have
\begin{equation} \label{der}
\frac{P_{n}(t)}{(1-t)^{n+1}} = \alpha_n t \left(\frac{P_{n-1}(t)}{(1-t)^{n}}\right)' + \beta_n \frac{P_{n-1}(t)}{(1-t)^n}. 
\end{equation}
This implies
\begin{equation}\label{lwcond}
P_n(t)= ((\alpha_n n- \beta_n)t + \beta_n) P_{n-1}(t) + \alpha_n t (1-t) P_{n-1}'(t),
\end{equation}
and $P_n(1)=n! \prod_{i=1}^n \alpha_i.$ By Proposition 3.5. in \cite{LW}, $P_n(t)$ has only real roots if and only if the leading coeffcient $\alpha_n \geq 0$. In that case, the following theorem applies.

\begin{theorem}\textnormal{(\cite{B})} \label{Ben}
Let $P_n(t)= \sum_{k} P_{n,k} t^k$ be a polynomial with all non-positive real roots, and $X_n$ be a random variable such that $\mathbf{P}(X_n=k)= \frac{P_{n,k}}{P_n(1)}.$ If  $\text{var}(X_n)$  goes to infinity as $n \rightarrow \infty$, then $X_n$ is asymptotically normal.
\end{theorem}
\noindent Pitman in \cite{Pit} gives an account of polynomials with only real zeros in this context. 

If $g_n(k)$ is not linear in $k,$ the real-rootedness does not necessarily hold (see Proposition \ref{prop}). Next, we identify the coefficients of recurrences for an arbitrary degree of $g_n.$ An example where $g_n(k)$ is of degree 2 is the number of descents in fixed-point free involutions, which we study in the next section. 

It is easy to derive the following from \eqref{rat}.
 \begin{equation}\label{inv}
P_{dn,k}= \sum_{i=0}^k (-1)^i \binom{d(n+1)+1}{i} f_n(k-i).
\end{equation}
Working the expression above, Koutras \cite{Kt} obtained the formula
\begin{equation}\label{kt}
P_{n+1,k}=(k \alpha_n + \beta_n)P_{n,k}+(\alpha_n(n-k+2)-\beta_n)P_{n,k-1}
\end{equation}
where $(x)_k=x(x-1)\cdots (x-k+1)$ is the falling factorial and  $f_n(k)=\alpha_n k +\beta_n.$ The special cases of the identities below are used in \cite{Kt} to break \eqref{inv} into parts. The first one is
\begin{equation}\label{a2}
(i)_l \binom{d(n+1)+1}{i} = (d(n+1)+d)_l \binom{dn+d-l+1}{i-l},
\end{equation}
and the second one is a consequence of Vandermonde's identity, 
\begin{equation}\label{a1}
 \binom{dn+d-l+1}{i-l} = \sum_{m=l}^{d} \binom{d-l}{m-l}\binom{dn+1}{i-m}.
\end{equation}
\noindent We refer to them later in the section.

 Now consider the operator $\Delta=\sum_{k\geq 0} \frac{(-1)^k}{k!}D^k$ (not to be confused with the difference operator), where $D$ is the usual differentiation. $\Delta^i$ means applying $\Delta$ $i$ times. It acts on monomials by shifting, i.e. $\Delta^i x^n = (x-i)^n,$ which can be shown inductively, e.g., in \cite{Fo}. At the same time, it can be interpreted as Taylor expansion, i.e.,
\begin{equation*}
\Delta^if(k)=\sum_{j=0}^{\infty}(-1)^j \frac{i^j}{j!}f^{(j)}(k).
\end{equation*}
Given that $f_{n+1}(k)=g_n(k)f_n(k)$ where $g_n(k)$ is a polynomial of degree $d,$ we have
\begin{align}
P_{d(n+1),k} &= \sum_{i=0}^k (-1)^i \binom{d(n+1)+1}{i} g_n(k-i) f_{n}(k-i)  \notag \\
 &= \sum_{i=0}^k (-1)^i \binom{d(n+1)+1}{i} \Delta^i g_n(k) f_{n}(k-i) \label{p2}
\end{align}
 by \eqref{inv}. In order to get use of \eqref{a2}, we change the basis as
\begin{equation*}
x^n=\sum_{k=0}^n \Stirling{n}{k} (x)_k,
\end{equation*}
where $\Stirling{n}{k}$ is Stirling number of the second type. Then, we have
\begin{align*}
\Delta^ig_n(k)&=\sum_{j=0}^d(-1)^j \frac{i^j}{j!}g_n^{(j)}(k) \\
&= \sum_{j=0}^d \frac{(-1)^j}{j!}g_n^{(j)}(k) \sum_{l=0}^j \Stirling{j}{l} (i)_l \\
&= \sum_{l=0}^d \left( \sum_{j=l}^d  \frac{(-1)^j}{j!}g_n^{(j)}(k)  \Stirling{j}{l} \right) (i)_l.
\end{align*}

Plugging the expression above in \eqref{p2} and applying \eqref{a1} and \eqref{a2} accordingly, we arrive at
\begin{equation} \label{recrec}
P_{d(n+1),k}= \sum_{m=0}^d (-1)^m \left( \sum_{l=0}^m (d(n+1)+1)_l \binom{d-l}{m-l} [(i)_l]\Delta^ig_n(k)  \right) P_{dn, k-m},
\end{equation}
where 
\begin{equation*}
[(i)_l]\Delta^ig_n(k)=\sum_{j=l}^d  \frac{(-1)^j}{j!}g_n^{(j)}(k)  \Stirling{j}{l}
\end{equation*}
is the coefficient of $(i)_l$ in $\Delta^ig_n(k).$ This gives a rather complicated but an explicit recurrence relation for combinatorial arrays defined in the rational form \eqref{rat} for which $f_{n+1}(k)/f_n(k)$ is a polynomial of $k.$ 


\section{Applications} \label{App}
The first two examples are simple cases, in the sense that they can be written as sums of independent random variables (which is shown in \cite{Fe}) by the same formulation for descents. 

\subsection{Inversions}

The number of inversions in a permutation $\pi$ is defined to be the number of all pairs $i \leq j$ such that $\pi(j) < \pi(i).$ Let $I_n$ be the random variable counting the number of inversions in a randomly chosen permutation from $S_n.$ Inserting $n+1$ in any possible $n+1$ positions and counting the probabilities, we obtain
\begin{equation*}
I_{n+1}= \left\lbrace I_n+i \text{ with probability } \frac{1}{n+1}, 0 \leq i \leq n \right\rbrace.
\end{equation*}

Then subtracting the mean of $I_n,$ we have the zero-mean martingale below.
\begin{equation*}
Z_n :=  I_n - \frac{\binom{n}{2} }{2}.
\end{equation*}
In the triangular array setting of Section \ref{clt}, the martingale difference $X_{n,i}=\frac{1}{s_n}(Z_{n,i}-Z_{n, i \text{-}1})$ satisfy
\begin{equation} \label{invdif}
X_{n,i}|\mathcal{F}_{n,i \text{-}1}=\frac{1}{s_n}\left\lbrace j-\frac{i-1}{2} \text{  with probability  } \frac{1}{i}, 0 \leq j \leq i-1 \right\rbrace,
\end{equation}
which are indeed independent random variables for $1 \leq i \leq n.$ Therefore, we can use classical limit theorems for the sums of independent random variables. The Lindeberg condition suffices to show the asymptotic normality (see Chapter 10 of \cite{Fe}).

We first calculate the variance of the differences,
\begin{equation*}
\mathbf{E}[X_{n,i}^2 | \mathcal{F}_{n,i \text{-}1}]= \frac{1}{s_n^2} \sum_{j=0}^{i-1} \frac{(j- \frac{i-1}{2})^2}{i }=\frac{1}{s_n^2}\frac{i^2-1}{12}.
\end{equation*}
Since $V_{n,n}^2=\sum_i \mathbf{E}[X_{n,i}^2 | \mathcal{F}_{n,i\text{-}1}]$ is deterministic in this case and its expected value is $1,$ we have 
\begin{equation*}
s_n^2=\sum_i  \frac{i^2-1}{12} = \frac{n(2n+5)(n-1)}{72}.
\end{equation*}
So that $|X_{n,i}| \leq C \frac{1}{\sqrt{n}},$ the Lindeberg condition \eqref{cona1} in Theorem \ref{thm2} can be verified by choosing $\epsilon>0$ small enough. It can also be shown that the absolute third moment is bounded, so the Berry-Esseen theorem for the sum of independent random variables gives an error term of order $n^{-1/2}.$

\subsection{Cycles}

Let $Q_n$ denote the number of cycles of a uniformly chosen random permutation in $S_n.$ Goncharov \cite{G} shows a central limit theorem for $Q_n$ considering it asymptotically as sum of Poisson distributions of the number of fixed length cycles. Suppose we insert $n+1$ either in any of the cycles of a given permutation, or place it as a fixed-point. This defines the zero-mean martingale
\begin{equation*}
Z_n=Q_n - \sum_1^n \frac{1}{k}.
\end{equation*}
Observe that $\mathbf{E}(Q_n) \sim \log n.$ Then
\begin{equation*}
X_{n,i}|\mathcal{F}_{n,i \text{-}1}=   \frac{1}{s_n}
    \begin{cases}
      1- \frac{1}{i}  &\text{ with prob. } \frac{1}{i},\\
      -\frac{1}{i}    &\text{ with prob. } \frac{i-1}{i}.
    \end{cases}
\end{equation*}
The conditional variance of the $i$th differences is
\begin{equation*}
\mathbf{E}[X_{n,i}^2 | \mathcal{F}_{n,i \text{-}1}]= \frac{1}{s_n^2} \left( \frac{i-1}{i^2} \right),
\end{equation*}
and 
\begin{equation*}
s_n^2= \sum_{i=1}^n \frac{i-1}{i^2} \sim \log n.
\end{equation*}
The Lindeberg condition is easily verified, since $|X_{n,i}| \leq \frac{1}{\sqrt{\log n}}.$ The asymptotic normality follows.
\subsection{Signed permutations}
The notion of descent in relation with group structure has its generalization in Coxeter groups (see Chapter 11 and 13 of \cite{PEN}). A particular case other than the symmetric group that we are interested is  Coxeter groups of type B, also known as hyperoctahedral groups. We denote them by $W_n$ for $n=1,2,\cdots$. They are isomorphic to the set of \textit{signed permutations} with usual multiplication in $S_n$.  A signed permutation $\pi$ is defined to be the mapping
\begin{equation*}
\pi : \{-n, \cdots,-1,0,1,\cdots,n \} \rightarrow \{-n, \cdots,-1,0,1,\cdots,n \}
\end{equation*}
which satisfies $\pi(-i)=-\pi(i).$ The descent set of $\pi$ is defined to be 
$ \{ 0 \leq i \leq n-1 : \pi(i)< \pi(i+1) \}.$ Unlike in the case of $S_n$, we also take the zeroth position into account. It is clear that the probability that $\pi$ has a descent at position $i$ is $\frac{1}{2}.$ Let $B_n$ count the number of descents of a random signed permutations. Then the linearity of expectation gives $\mathbf{E}(B_n)=\frac{n}{2}.$ Provided that 
\begin{equation*}
B_n(t):= \sum_{\pi \in W} t^{\textnormal{des}(\pi)}=\sum_{k=1}^n B_{n,k} t^k,
\end{equation*}
we have (Theorem 13.3. of \cite{PEN})
\begin{equation*}
\frac{B_n(t)}{(1-t)^{n+1}} = \sum_{k\geq 0} (2k+1)^n t^k,
\end{equation*}
and also the recurrence relation
\begin{equation*}
B_{n+1,k}=(2k+1)B_{n,k}+(2n-2k+3)B_{n,k-1}.
\end{equation*}
It satisfies the condition \eqref{diag}, so we have the submartingale below. 
\begin{equation*} 
B_{n+1}=\begin{aligned}
   &   
   \begin{cases}
     B_{n},&  \text{with prob. } \frac{2B_n+1}{2n+2},\\
       B_{n}+1,& \text{with prob. } \frac{2n-2B_n+1}{2n+2}.
    \end{cases}
\end{aligned}
\end{equation*}
After normalization,  $Z_n=n \left(B_n-\frac{n}{2}\right)$ is a zero-mean martingale. 
In order to derive the moments, we first use \eqref{lwcond} and obtain
\begin{equation*}
B_n(t)=(1+(2n-1)t)B_{n-1}(t) + 2t(1-t)B_{n-1}'(t).
\end{equation*}
Then, we can use method of moments observing that
\begin{equation*}
\mathbf{E}(B_n)=\frac{B_n'(1)}{B_n(1)}.
\end{equation*}
Similarly,
\begin{equation*}
\mathbf{E}(B_n^2)=\frac{B_n''(1)+B_n'(1)}{B_n(1)}.
\end{equation*}
We do not carry out the calculations, but point out that the leading terms of the moments are asymptotically same with the classical Eulerian numbers'. Consider the generating function
\begin{equation*}
B(u,t)=\sum_{n \geq 0} B_n(t) \frac{u^n}{n!} = \frac{(t-1)e^{u(t-1)}}{t-e^{2u(t-1)}},
\end{equation*}
which can be found in Chapter 13 of \cite{PEN}. It has a simple pole at $r(t)=\frac{\log t}{2(t-1)},$ which is a constant multiple of the simple pole of $A(u,t),$ the generating function for Eulerian numbers (see Section 9.6. of \cite{FS} for an accessible singularity analysis for combinatorial arrays). Then, Theorem 1 in \cite{B} verifies the claim.  

In fact, Eulerian polynomials for all Coxeter groups have only real roots, so that Theorem \ref{Ben} applies. The last open case was for Coxeter groups of type D, which is shown in \cite{SV}. A simple recurrence relation for Eulerian numbers for Coxeter groups of type $D$ is not available to the best of author's knowledge. But their relation to the first two types' is rather simple,
\begin{equation*}
D_{n,k} = B_{n,k} - n 2^{n-1}A_{n,k-1}.
\end{equation*}

\subsection{Stirling permutations}
Gessel and Stanley define Stirling polynomials, $f_n(k)=\Stirling{n+k}{k},$ and their generating function
\begin{equation*}
\frac{C_n(t)}{(1-t)^{n+1}} = \sum_{k\geq 0} \Stirling{n+k}{k} t^k,
\end{equation*}
in \cite{GS}, and they provide a combinatorial interpretation for $C_n(t).$ The coefficient of $t^k$ in $C_n(t),$ call it $C_{n,k},$ is the number of \textit{Stirling permutations} with exactly $k$ descents. Stirling permutations are permutations of the multiset $\{1,1,2,2,\cdots,n,n\}$ such that the numbers between two occurences of $i$ are larger than $i$ for all $1\leq i \leq n.$  The numbers $\{C_{n,k}\}$, known as second-order Eulerian numbers, appear in different branches of combinatorics (see \cite{HV}).

By \eqref{stir}, it can be easily shown that
\begin{equation} \label{recc}
C_{n}(t)=(2n-1)t C_{n-1}(t) + t(1-t)C_{n-1}'(t).
\end{equation}
Let $C_n$ be the random variable counting descents in Stirling permutations. Then working the coefficients, we obtain the submartingale

\begin{equation*} 
C_{n+1}=\begin{aligned}
   &   
   \begin{cases}
     C_{n},&  \text{with prob. } \frac{C_n}{2n+1},\\
       C_{n}+1,& \text{with prob. } \frac{2n-C_n +1}{2n+1}.
    \end{cases}
\end{aligned}
\end{equation*}
As in the previous example, the moments can be studied by \eqref{recc}. Note that the number of Stirling permutations is $C_n(1)=(2n-1)!!\equiv 1 \cdot 3 \cdots (2n-1).
$  For the first moment, we have
\begin{equation*}
\mathbf{E}(C_n)=1+ \frac{2n-2}{2n-1} \, \mathbf{E}(C_{n-1}),
\end{equation*}
which does not yield a simple expression. But a purely combinatorial count for the first moment by B\'ona leads him to this curious identity
\begin{equation*}
\mathbf{E}(C_n)=\sum_{k=0}^{n-1} \prod_{i=1}^k \frac{2n-2i}{2n-2i+1} = \frac{2n+1}{3}
\end{equation*}
in \cite{BoS}. It is also shown therein that $C_n(t)$ has only real zeros. Also, observe that \eqref{recc} satisfies \eqref{lwcond} with a positive coefficient. Therefore, Theorem \ref{Ben} is applicable to show asymptotic normality. 
 
\subsection{Alternating runs}
For $\pi \in S_n,$ $\pi$ is said to change direction at position $i,$ if either $ \pi(i-1) > \pi(i) < \pi(i+1)$ or $ \pi(i-1) < \pi(i) > \pi(i+1).$Then $\pi$ has $k-$\textit{alternating runs} if there exist exactly $k-1$ positions that it changes direction (see Section 1.2 of \cite{Bo}). Let $G_{n,k}$ be the number of k-alternating runs, then it is known to satisfy the recurrence relation


\begin{equation} \label{run}
G_{n+1,k}= k G_{n,k} + 2 G_{n,k-1} + (n-k+1) G_{n,k-2}.
\end{equation}
\begin{equation*} 
G_{n+1}=\begin{aligned}
   &   
   \begin{cases}
     G_{n},&  \text{with prob. } \frac{G_n}{n+1},\\
       G_{n}+1,& \text{with prob. } \frac{2}{n+1},\\
       G_{n}+2,&  \text{with prob. } \frac{n-G_{n}-1}{n+1}.
    \end{cases}
\end{aligned}
\end{equation*}
We can calculate the first moment simply by observing that the probability of $\pi$ to change direction at each position $i$ is $\frac{2}{3}.$ We have 
$\mathbf{E}(G_n)=\frac{2}{3}(n-2)+1$, whence we have the zero-mean martingale
\begin{equation*}
Z_n=\frac{n(n-1)}{2} \left( G_n - \frac{2n-1}{3} \right).
\end{equation*}
Another important statistic, which is closely related to alternating runs is the length of the longest alternating subsequences. Let $L_n$ be the corresponding random variable. We have the simple relation \cite{St},
\begin{equation} \label{alt}
L_{n,k} = \frac{1}{2}(G_{n,k-1}+G_{n,k}).
\end{equation}
Putting \eqref{run} and \eqref{alt} together, we have
\begin{equation*}
L_{n+1,k}= k L_{n,k} + L_{n,k-1} + (n-k+2) L_{n-1,k-2},
\end{equation*}
which defines the submartingale
\begin{equation*} 
L_{n+1}=\begin{aligned}
   &   
   \begin{cases}
     L_{n},&  \text{with prob. } \frac{L_n}{n+1},\\
       L_{n}+1,& \text{with prob. } \frac{1}{n+1},\\
       L_{n}+2,&  \text{with prob. } \frac{n-L_n}{n+1}.
    \end{cases}
\end{aligned}
\end{equation*}
Similarly, $Z_n=\frac{n(n-1)}{2} \left( L_n - \frac{4n+1}{6} \right)$ is a zero-mean martingale. In the setting of Section \ref{Bolt}, the martingale difference $X_i=Z_{i}-Z_{i-1}$ satisfies
\begin{equation*} 
X_{i}|\mathcal{F}_{i \text{-}1}= \frac{i-1}{2}\begin{aligned}
   &   
   \begin{cases}
     2L_{i-1}-2i+1,&  \text{with prob. } \frac{L_{i-1}}{i},\\
       2L_{i-1}-i+1,& \text{with prob. } \frac{1}{i},\\
       2L_{i-1}+1,&  \text{with prob. } \frac{i-L_{i-1}-1}{i}.
    \end{cases}
\end{aligned}
\end{equation*}
Different techniques are used to show the asymptotic normality \cite{Ho, St, WS,W} of $L_n$. In order to study the convergence rate in the limit, the higher moments of $L_n$ can be computed from the factorial moment generating function given in \cite{St} to check if the conditions of Theorem \ref{convthm} are satisfied.

\subsection{Matchings} \label{match}

A \textit{matching} is a fixed-point free involution, a permutation consisting only of transpositions, in $S_{2n}.$ The asymptotic normality of the number of descents in matchings is addressed by Kim \cite{K}. It is proven by pointwise convergence of the moment generating function. We obtain the same result by martingales. Let 
\begin{equation*}
J_{2n}(t):= \sum_{\pi \in J_{2n}}t^{\textnormal{des}(\pi)} = \sum_{k=0}^{2n-1} J_{2n,k} t^k.
\end{equation*}
D\'esarm\'enien and Foata \cite{DF} derived the generating function below using algebraic properties of Schur functions. 
\begin{equation*}
\sum_{n \geq 0}J_n(t)\frac{u^n}{(1-t)^{n+1}} = \sum_{k\geq 0} \frac{t^k}{(1-u^2)^{\binom{k+1}{2}}}.
\end{equation*}
Note that $J_{n}(t)=0$ if $n$ is odd, and $J_{2n}(1)=(2n-1)!!.$ Equating the coefficients of $u^{2n}$ gives
\begin{equation*}
\frac{J_{2n}(t)}{(1-t)^{2n+1}}= \sum_{k \geq 0} \binom{\binom{k+1}{2}+n-1}{n} t^k.
\end{equation*}
It is tempting to ask if Theorem \ref{Ben} is applicable in this case. The answer is negative by the following proposition.

\begin{proposition}\label{prop}
$J_{2n}(t)$ is not log-concave.
\end{proposition}
\pf
By \eqref{inv} we have
\begin{equation*}
J_{2n}(t)= \sum_{k=0}^{2n-1} J_{2n,k} t^k = \sum_{k=0}^{2n-1}\left( \sum_{i=0}^k (-1)^{i} \binom{2n+1}{i} \binom{\binom{k-i+1}{2}+n-1}{n} \right)t^k .
\end{equation*}

By inverting the coefficients (see Chapter 3.3 of \cite{R}), 
\begin{equation*}
a_{n,k}:=\binom{\binom{k+1}{2}+n-1}{n}=\sum_{i=0}^k \binom{2n+i}{i} J_{2n,k-i}.
\end{equation*}
Then, we have
\begin{equation*}
\sum_{k}a_{n,k} x^k =\left(\sum_{i} \binom{2n+i}{i} x^i \right) \left(\sum_{j} J_{2n,j} x^j \right).
\end{equation*}
 The first three terms of $\{a_{n,k}\}$ reveal below that it is not a log-concave sequence since
\begin{equation*}
a_{n,2}^2=\binom{n+2}{2}^2 < \binom{n+6}{2}=a_{n,1}a_{n,3}.
\end{equation*}

On the other hand, the coefficients of the first sum on the right hand side is a log-concave sequence. But Proposition 1 in \cite{Su} says that the product of two log-concave polynomials is also log-concave, which implies $\{J_{2n,k}\}$ is not log concave. 
\bbox \\

Since real-rootedness implies log-concavity, the roots of $J_{2n}(t)$ are not real-only. So Theorem \ref{Ben} is not applicable in this case. Nevertheless, the asymptotic normality is shown by Kim using generating function to determine the asymptotic behaviour of the moments. We give a martingale proof below.
The recurrence relation is derived in \cite{GZ}, it can also be obtained from \eqref{recrec}.
\begin{align*}
J_{2n+2,k}=& \frac{k(k+1)+2n}{2n+2}J_{2n,k}+\frac{2(k-1)(2n-k+1)+2}{2n+2}J_{2n,k-1}\\
&+\frac{(2n-k+2)(2n-k+3)+2n}{2n+2}J_{2n,k-2}.
\end{align*}
Let $M_{2n}$ count the number of descents in random matchings. We can verify \eqref{diag} and define the submartingale 
\begin{equation*} 
M_{2n+2}=\begin{aligned}
   &   
   \begin{cases}
     M_{2n},&  \text{with prob. } \frac{M_{2n}(M_{2n}+1)+2n}{(2n+1)(2n+2)},\\[5pt]
       M_{2n}+1,& \text{with prob. } \frac{2M_{2n}(2n-M_{2n})+2}{(2n+1)(2n+2)},\\[5pt]
       M_{2n}+2,&  \text{with prob. } \frac{(2n-M_{2n})(2n-M_{2n}+1)+2n}{(2n+1)(2n+2)}.
    \end{cases}
\end{aligned}
\end{equation*}
We have $\mathbf{E}(M_{2n+2}|\mathcal{F}_{n})=\frac{n}{n+1}M_{2n} + \frac{2n+1}{n+1}.$ Note that $E(M_{2n})=n$ (see \cite{K} for its derivation). Then by properly scaling, we have the zero-mean martingale $Z_n= n \left(M_{2n} - n \right).$ We then have 
\begin{equation*}
X_{i}|\mathcal{F}_{i \text{-}1}=\begin{aligned}
   &   
   \begin{cases}
     W_{i\text{-}1} - i,&  \text{with prob. } \frac{(W_{i\text{-}1} +i-1)(W_{i\text{-}1} +i)+2i-2}{2i(2i-1)},\\[5pt]
       W_{i\text{-}1},& \text{with prob. } \frac{2(W_{i\text{-}1} +i-1) (i-1-W_{i\text{-}1})+2}{2i(2i-1)},\\[5pt]
       W_{i\text{-}1} + i,&  \text{with prob. } \frac{(i-1-W_{i\text{-}1})(i-W_{i\text{-}1})+2i-2}{2i(2i-1)},
    \end{cases}
\end{aligned}
\end{equation*}
where $W_i=M_{2i}-i.$ Next we calculate the higher confitional moments to apply Theorem \ref{convthm}. 
\begin{align*}
\mathbf{E}[X_{i}^2 | \mathcal{F}_{i\text{-}1}] &=  \frac{i-1}{2i-1}\left( i^2+2i-W_{i-1}^2 \right),\\
\mathbf{E}[X_{i}^3 | \mathcal{F}_{i\text{-}1}]&= \frac{2-i}{2i-1} W_{i-1}^3 + \frac{i(i^2-2i-6)}{2i-1} W_{i-1}, \\
\mathbf{E}[X_{i}^4 | \mathcal{F}_{i\text{-}1}] &= \frac{3}{2i-1} W_{i-1}^4 -\frac{i(i^2+2i+12)}{2i-1} W_{i-1}^2+\frac{i^3(i+1)(i-2)}{2i-1}.
\end{align*}

The moments of $M_{2i}$ up to the fifth degree are calculated in \cite{Ki}. We can compute the central moments using them to show that the conditions of Theorem \ref{convthm} is satisfied. 

Matchings define a particular conjugacy class in $S_{2n},$ namely the one with $n$ 2-cycles. In general, let $C$ be a conjugacy class in $S_n$ with $n_i$ $i-$cycles and $A_C(t)= \sum_{\pi \in C} t^{\textnormal{des}(\pi)}.$ It can be found in \cite{Fu} that
\begin{equation*}
\frac{A_C(t)}{(1-t)^{n+1}}=\sum_{k\geq 0}t^k \prod_{i=1}^n \binom{f_{ik}+n_i-1}{n_i}.
\end{equation*}
where $f_{ik}=\frac{1}{i}\sum_{d|i} \mu(d) k^{i/d}$ and $\mu(\cdot)$ is the M\"{o}bius function. We are interested in finding recursive formulae for the number of permutations in a fixed conjugacy class with a given number of descents. These numbers arise in riffle shuffles \cite{DMP}. A recursive formula can be derived from \eqref{recrec} for certain cases.  For example, it is possible to find such recurrence relation in the case of $n$ $3-$cycles in $S_{3n},$ but it is not possible for involutions with no restriction on the number of fixed-points. But if the number of fixed-points is fixed at a certain proportion, then a martingale formulation is possible. 
\subsection{Two-sided Eulerian numbers}
The last example is a vector descent statistic which can be studied by multivariate martingale limit theorems. Two-sided Eulerian numbers are introduced in \cite{Car} as the coefficients of   
\begin{equation*}
A_n(t,s)=\sum_{\pi \in S_n} t^{\textnormal{des}(\pi)+1} s^{des(\pi^{-1})+1} .  
\end{equation*}
The generating function is obtained by counting $n$ unlabelled balls in $kl$ distinct boxes, which is a two-dimensional analogue of barred permutations. See \cite{Pe} for the counting argument and a survey on these numbers. It is given by
\begin{equation*}
\frac{A_{n}(t,s)}{(1-t)^{n+1}(1-s)^{n+1}}=\sum_{k,l \geq 0} \binom{kl+n-1}{n} t^k s^l.
\end{equation*}
The recurrence relation on coefficients are derived in \cite{Car}; it satisfies the two-dimensional generalization of \eqref{diag}. Let $D_n$ be the random variable counting the number of descents of a uniformly chosen permutation $\pi,$ while $D'_n$ counts the number of descents in $\pi^{-1}.$ Then,
\begin{equation*} 
(D_{n+1}, D'_{n+1})=\begin{aligned}
   &   
   \begin{cases}
     (D_n,D'_n) ,&  \text{with prob. } \frac{(D_n+1)(D'_n+1)+n}{(n+1)^2},\\
       (D_n+1,D'_n),& \text{with prob. } \frac{(n-D_n)(D'_n+1)-n}{(n+1)^2},\\
       (D_n,D'_n+1),&  \text{with prob. } \frac{(D_n+1)(n-D'_n)-n}{(n+1)^2} \\
       (D_n+1, D'_n+1), & \text{with prob. } \frac{(n-D_n)(n-D'_n)+n}{(n+1)^2}
    \end{cases}
\end{aligned}
\end{equation*}
is a submartingale, and 
\begin{equation*}
(Z_n,Z'_n)=n\left( D_n-\frac{n-1}{2}, D'_n-\frac{n-1}{2} \right)
\end{equation*}
is  a zero-mean martingale. A central limit theorem is recently shown in \cite{CD}. We use below a multivariate limit theorem for martingales, whose proof is  in \cite{Aa}. The theorem is in functional form, but we can embed $(Z_n, Z'_n)$ in Brownian motion (see Appendix of \cite{HH}), apply the theorem, and the unit time for standard Brownian motion gives standard normal distribution. Later, equivalent conditions for the theorem are given in \cite{He}. They are similar to one-dimensional case. First, we need to verify Lindeberg-condition and the convergence of conditional variance for both coordinates. Since $D_n$ and $D'_n$ are identically distributed, and the case for $D_n$ is already covered in Section \ref{clt}, we only need to show the additional condition on covariances,
\begin{equation}\label{cova}
\sum_{i}\mathbf{E}[X_{n,i}X'_{n,i}|\mathcal{F}_{n,i-1}] \overset{p}{\to} 0,
\end{equation}
where $X_{n,i}$ and $X'_{n,i}$ are defined as in Section \ref{clt} to be martingale differences. We also define the central random variables, $W_i=D_i-\frac{i-1}{2}$ and $W'_i=D'_i - \frac{i-1}{2}.$ It can be calculated from \eqref{wi} that
\begin{equation*}
\mathbf{E}[X_{n,i}X'_{n,i}|\mathcal{F}_{n,i-1}]=\frac{36}{n^2(n+1)} W_{i-1}W'_{i-1}.
\end{equation*}
In order to use Chebyschev's inequality to show \eqref{cova}, we first prove the following lemma.

\begin{lemma}\label{covlem} For a uniformly chosen permutation $\pi \in S_n$, define the random variables, $D_n(\pi)=\textnormal{des}(\pi)$ and $D'_n(\pi)= des(\pi^{-1})$. Then, $\mathbf{E}\left[\left(D_n-\frac{n-1}{2}\right)^2 \left(D'_n-\frac{n-1}{2}\right)^2\right]$ is of order $n^2.$
\end{lemma}
\pf
Define $D_n=\sum_{i=1}^{n-1} T_i$ and $D'_n=\sum_{i=1}^{n-1} S_i$ as in Section \ref{Mom}, where $T_i$ is as before and  
\begin{equation*}
   S_i(\pi) = 
    \begin{cases}
      1  \text{ if } \pi^{-1}(i) > \pi^{-1}(i+1),\\
      0 \text{ otherwise.} 
    \end{cases}
\end{equation*}
We first evaluate
\begin{equation*}
\mathbf{E}(D_n D'_n)= \sum_{i=1}^{n-1} \sum_{j=1}^{n-1} \mathbf{E}(T_i S_j).
\end{equation*}
As suggested in \cite{CD}, we define $K=\{\pi(i), \pi(i+1) \} \cap \{j,j+1 \},$ and break the expectation into three terms with respect to the size of $K$,
\begin{align*}
\mathbf{E}(T_i S_j)&= \mathbf{P}(T_iS_j=1 |K=0) \cdot \mathbf{P}(K=0)\\ &+ \mathbf{P}(T_iS_j=1 |K=1) \cdot \mathbf{P}(K=1) \\ &+\mathbf{P}(T_iS_j=1 |K=2) \cdot \mathbf{P}(K=2).
\end{align*}
It is straightforward that $\mathbf{P}(T_i=S_j=1 |K=0)=\frac{1}{4}$ by independence. The same holds true for $K=1$ by symmetry. Suppose $\pi(i)=j,$ which has  $\frac{1}{4}$ probability conditioned on $K=1.$ Then $T_iS_j=1$ if and only if $\pi(i+1) < j$ and $\pi^{-1}(j+1)<i.$ While if $\pi(i+1)=j$, $T_iS_j=1$ if and only if $\pi(i) > j$ and $\pi^{-1}(j+1)<i.$ Considering the other two cases as well, we observe that the probabilities of $\{T_iS_j=1\}$ add up to 1 through four equiprobable cases. Therefore, $\mathbf{P}(T_i=S_j=1 |K=1)=\frac{1}{4}.$ Finally, we deal with the last term. There are four possible pairs, and two of them give descents in both positions, so  $\mathbf{P}(T_i=S_j=1 |K=2)=\frac{1}{2}.$ A simple counting argument shows that $\mathbf{P}(K=2)=\frac{1}{\binom{n}{2}}=\frac{2}{n(n-1)}.$ Hence,
\begin{equation*}
\mathbf{E}(T_i S_j)= \frac{1}{4}+\frac{1}{2n(n-1)},
\end{equation*}
which eventually shows that
\begin{equation*}
\mathbf{E}(D_n D'_n)-\mathbf{E}(D_n) \mathbf{E}(D'_n) =\frac{n-1}{2n},
\end{equation*}
which is of constant order. We argue for $\mathbf{E}(D^2_n D'_n)$ and $\mathbf{E}(D^2_n D'^2_n)$ along the same lines. Consider
\begin{equation*}
\mathbf{E}(D_n^2 D'_n)=\sum_{i=1}^{n-1}\sum_{j=1}^{n-1} \sum_{k=1}^{n-1} \mathbf{E}(T_i T_j S_k),  
\end{equation*}
(We refer to Section \ref{Mom} for treatment of cases where $j=i$ and $j=i+1$) Again, depending on the size of $\{\pi(i), \pi(i+1),\pi(j),\pi(j+1)\}  \cap \{ k, k+1\},$ we can perform case-by-case analysis. Given that the size of the intersection is 0 or 1, it can be shown that $T_i T_j$ is independent of $S_k.$ Since, the probability that the size of the intersection is 2, is of order $\frac{1}{n^2}$ and there are $(n-1)^3$ terms, we have 
\begin{equation*}
\mathbf{E}(D^2_n D'_n)-\mathbf{E}(D^2_n) \mathbf{E}(D'_n) =\mathcal{O}(n).
\end{equation*}
A similar argument gives
\begin{equation*}
\mathbf{E}(D^2_n D'^2_n)-\mathbf{E}(D^2_n) \mathbf{E}(D'^2_n) =\mathcal{O}(n^2).
\end{equation*}
Therefore, we can check term by term that
\begin{equation*}
\mathbf{E}(W_n^2 W'^2_n) - \mathbf{E}(W_n^2)\mathbf{E}(W'^2_n) = \mathcal{O}(n^2).
\end{equation*}
\bbox

We have
\begin{align*}
\mathbf{E}\left(\sum_{i}\mathbf{E}[X_{n,i}X'_{n,i}|\mathcal{F}_{n,i-1}] \right)^2 & 
\leq \frac{C}{n^6} \sum_{1 \leq i,j\leq n\text{-}1}  \mathbf{E}(W_{i-1}W'_{i-1}W_{j-1}W'_{j-1}) \\
 & \leq \frac{C}{n^6}\sum_{1 \leq i,j\leq n\text{-}1} \sqrt{\mathbf{E} ( W_{i \text{-} 1}^2 W'^2_{i \text{-} 1}) \mathbf{E}( W_{j \text{-} 1}^2 W'^2_{j \text{-} 1})} \\
  &\leq  \frac{C}{n^6} \sum_{1 \leq i,j\leq n\text{-}1}  i j \\
& \leq \frac{C }{n^2} \rightarrow 0,
\end{align*}
where the third inequality follows from Lemma \ref{covlem}. Therefore, we verify \eqref{cova} by Chebyschev's inequality. Then the conditions of Theorem 3.3 in \cite{He} are satisfied. We conclude that 
\begin{equation*}
\left( \frac{D_n - \frac{n-1}{2}}{\sqrt{\frac{n+1}{12}}},\frac{D'_n - \frac{n-1}{2}}{\sqrt{\frac{n+1}{12}}}\right)
\end{equation*}
 is asymptotically bivariate normal with zero mean and unit covariance matrix.

\section{Proof of Theorem \ref{convthm}} \label{Pf}
In the rest of the paper, we repeatedly use $C, c, C'$ and $c'$ for different constants that are all independent of $n$, if no confusion arises.

\subsection{Preliminaries} The following two lemmas play key roles in the Bolthausen's proof. 
\begin{lemma} \textnormal{(\cite{Bh})}\label{boldlem2}
Let $X$ and $\xi$ be random variables and $\Phi$ be the standard normal distribution function. Define
\begin{equation*}
\delta=\sup_t |\mathbf{P}(X \leq t) - \Phi(t)|
\end{equation*}
\begin{equation*}
\textnormal{ and } \delta^{*}=\sup_t |\mathbf{P}(X + \xi\leq t) - \Phi(t)|.
\end{equation*}
\end{lemma}
Then,
\begin{equation*}
\delta \leq 2 \delta^{*} + \frac{5}{\sqrt{2 \pi}} \| \mathbf{E}(\xi^2 | X)  \|^{1/2}_{\infty}.
\end{equation*}

\begin{lemma} \textnormal{(\cite{Bh})} \label{boldlem3}
Let $X$ be a random variable, $f \in \mathcal{L}_1(\mathbb{R})$ be of bounded variation $\| f\|_V$ and $\delta$ be as defined in Lemma \ref{boldlem2}. For $a\neq 0,$
\begin{equation*}
|\mathbf{E}[f(aX+b)]| \leq \|f\|_V \, \delta + \frac{\|f\|_1}{\sqrt{2\pi}}\, |a|^{-1}.
\end{equation*}
\end{lemma}

We next discuss some technically useful implications of the additional assumption \eqref{maincond} to the Berry-Esseen theorem in \cite{Bh}. First, \eqref{maincond} implies that for $\epsilon> 0$ there exists $N \in \mathbb{N}$ such that $m \geq N$ implies 
\begin{equation}\label{boundondelta}
\frac{1-\epsilon}{m} \leq \frac{\sigma_{m}^2}{s_m^2} \leq \frac{C+\epsilon}{m}
\end{equation}
for some constant $C$. Let us fix $\alpha \in (0,1)$ and choose $n$ large enough that $ \alpha n > N.$ Then since $s_n^2=s_{n-1}^2+\sigma_n^2,$ we have $\frac{s_{n-1}^2}{s_n} \leq 1-\frac{1-\epsilon}{n}$ by \eqref{boundondelta}. Repeating the same argument, we have
\begin{equation}\label{varratio}
\frac{s_m^2}{s_n^2} \leq \prod_{k=m+1}^n \left(1-\frac{1-\epsilon}{k}\right).
\end{equation}
Since 
\begin{equation*}
\log \left(1-\frac{1-\epsilon}{k}\right)= -\frac{1-\epsilon}{k} + \mathcal{O}(k^{-2}),
\end{equation*}
comparing the right-hand side of \eqref{varratio} with $\prod_{k=m+1}^n \left(1-\frac{1}{k}\right) = \frac{m}{n},$ we have
\begin{equation}\label{varratiorange}
\frac{s_m^2}{s_n^2} \leq \prod_{k=m+1}^n \left(1-\frac{1-\epsilon}{k}\right) \leq \left(\frac{m}{n}\right)^{1-\epsilon}+ \mathcal{O}(n^{-1})
\end{equation}
for $m \geq \alpha n.$ In addition, reconsidering \eqref{varratio},
 \begin{equation}\label{sdratio}
\sigma_m^2 \leq (C+\epsilon) \frac{s_m^2}{m} \leq {\alpha}^{-\epsilon}(C+\epsilon) \frac{s_n^2}{n}
\end{equation} 
provided that $m \geq \alpha n.$ If we take $m \leq \alpha n,$ it follows from \eqref{varratiorange} that 
\begin{equation}\label{1tom}
s_n^2-s_m^2 \geq s_n^2 - s_{\alpha n}^2 \geq (1-\alpha^{1-\epsilon})s_n^2.
\end{equation}
\vspace*{1mm}

The proof will also use the following lemma in relation to \eqref{maincond}.
\begin{lemma}\label{lemma3}
Let $\gamma >0$ and $s\neq 1.$ Then
\begin{equation*}
\sum_{k=1}^n \left(\frac{1}{k+ \gamma}\right)^s \leq C_s \max\{1, n ^{1-s}\}
\end{equation*}
for some constant $C_s$ independent of $n$.
\end{lemma}
\pf Suppose $s >1.$ Then 
\begin{equation*}
\sum_{k=1}^n \left(\frac{1}{k+ \gamma}\right)^s \leq \sum_{k=1}^n \frac{1}{k^s}  \leq \zeta(s) 
\end{equation*}
where $\zeta$ is Riemann zeta function. If $s < 1$, we can bound it by the upper Riemann sums of $\int_{1}^n x^{-s} dx,$ which is $\frac{n^{1-s}}{1-s}$ plus smaller order terms.
\bbox

\subsection{Lindeberg's argument} Let $X_1, \dots, X_n$ be martingale differences with $\sigma_i^2=\mathbf{E}[X_i^2]$ and $s_k^2=\sum_{i=1}^k \sigma_i^2.$ Define $Z_1, \dots, Z_n, \xi$ to be central normal variables with $\textnormal{Var}(Z_i) = \sigma_i^2$ and $\textnormal{Var}(\xi)=\kappa^2 >0,$ which are independen, both from $X_1,\dots, X_n$ and also among themselves. Later in the proof, $\kappa^2$ is chosen to be a function of $s_n^2.$ Define
\begin{align*}
\delta(i)&=\sup_t \left| \mathbf{P}\left(\frac{X_1+\cdots+X_i}{s_i} \leq t\right) - \Phi(t) \right| \\
\delta_{\xi}(i)&=\sup_t \left| \mathbf{P}\left(\frac{X_1+\cdots+X_i+\xi}{s_i} \leq t\right) - \Phi(t) \right|.
\end{align*}
Lemma \ref{boldlem2} and the simple estimate below
\begin{equation*}
\sup_t |\Phi(t) - \Phi(\alpha t)| \leq \frac{e^{-1/2}}{\sqrt{2\pi}} \, |\alpha -1|
\end{equation*}
imply that
\begin{equation} \label{bounddelta}
\delta(n) \leq 4 \delta_{\xi}(n) + C\frac{\kappa}{s_n}.
\end{equation}
Our goal is to bound
\begin{align*}
\delta_{\xi}(n)=& \sup_t \left|\mathbf{P}\left(\frac{X_1+\cdots + X_n+ \xi}{s_n} \leq t\right) - \mathbf{P}\left(\frac{Z_1+\cdots + Z_n+ \xi}{s_n} \leq t\right)\right|. 
\end{align*}
Lindeberg's idea in his proof of the central limit theorem \cite{L} is to write the difference of probabilities above as
\begin{align*}
& \sum_{i=1}^n\mathbf{P}\left(\frac{X_1+\cdots + X_{i-1}+X_i+Z_{i+1}+\cdots+Z_n+ \xi}{s_n} \leq t\right)
 \\
  - &\sum_{i=1}^n\mathbf{P}\left(\frac{X_1+\cdots + X_{i-1}+Z_i+Z_{i+1}+\cdots+Z_n+ \xi}{s_n} \leq t\right).
\end{align*}

Define $U_m = \sum_{i=1}^{m-1}X_i /s_n.$ Observe that $\sum_{i=m}^{n}Z_i+\xi /s_n$ is a central normal random variable with variance $\lambda_m^2 :=\frac{s_n^2-s_m^2 + \kappa^2}{s_n^2}$ and independent of $U_m, X_m$ and $Z_m.$ Write the above expression as
\begin{equation*}
\sum_{m=1}^n \Phi\left( \frac{t-U_m}{\lambda_m}-\frac{X_m}{\lambda_m s_n}\right) - \Phi\left( \frac{t-U_m}{\lambda_m}-\frac{Z_m}{\lambda_m s_n}\right),
\end{equation*}
then to apply martingale properties and the fact that $U_m$ is $\mathcal{F}_{m-1}$ measurable, we further write it as
\begin{equation} \label{lindo}
 \sum_{m=1}^n \mathbf{E}\left(\mathbf{E} \left[ \Phi \left( \frac{t-U_m}{\lambda_m} - \frac{X_m}{\lambda_m s_n}\right) - \Phi \left( \frac{t-U_m}{\lambda_m} - \frac{Z_m}{\lambda_m s_n}\right) \bigg| \mathcal{F}_{m-1}\right] \right).
\end{equation}
The idea is to use the Taylor expansion of the normal distribution function to obtain bounds on the sum above.

\subsection{Error term in the Taylor expansion} The estimations below are for a fixed term of the sum, so we drop the index for $m$ and define 
$$U=\frac{t-U_m}{\lambda_m}, \, V_x=\frac{X_m}{\lambda_m s_n} \textnormal{ and } V_z=\frac{Z_m}{\lambda_m s_n}.$$ 
Consider the Taylor expansion of the third order,
\begin{align*}
\Phi(u-v) &= \Phi(u) - v \varphi(u)+ \frac{v^2}{2} \varphi'(u) - \frac{v^3}{6} \varphi''(u) + R_4(u,v),
\end{align*}
where $R_4(u,v)$ is the remainder term. So the expression in the expectation of the $m$th term of \eqref{lindo} is
\begin{align*}
&\mathbf{E} \left[ \Phi \left( U-V_x\right) - \Phi \left( U-V_z\right) | \mathcal{F}_{m-1}\right] \\ 
&=\varphi(U)\mathbf{E}[(V_z-V_x) | \mathcal{F}_{m-1}] + \varphi'(U) \mathbf{E} \left[\left( \frac{V_x^2-V_z^2}{2} \right)  \bigg| \mathcal{F}_{m-1}\right] \\
&- \varphi''(U) \mathbf{E} \left[\left( \frac{V_x^3-V_z^3}{6} \right) \bigg| \mathcal{F}_{m-1}\right]+ \mathbf{E}[R_4(U,V_x)| \mathcal{F}_{m-1}] +  \mathbf{E}[R_4(U,V_z)| \mathcal{F}_{m-1}].
\end{align*}
Observe that the first term vanishes since $V_x$ is a martingale difference. For the third term, we note that $\mathbf{E}(Z_m^3)=0.$ Then we write  
\begin{equation}\label{a1toa4}
\mathbf{E} \left[ \Phi \left( U-V_x\right) - \Phi \left( U-V_z\right) | \mathcal{F}_{m-1}\right] = A_1 + A_2 + A_3 + A_4
\end{equation}
where
\begin{align*}
A_2 &= \frac{1}{2 \lambda_m^2 s_n^2} \mathbf{E}((\mathbf{E}[X_m^2|\mathcal{F}_{m-1}]-\sigma_m^2)\varphi'(U)),\\
A_3 & = -\frac{1}{6 \lambda_m^3 s_n^3} \mathbf{E}(\mathbf{E}[X_m^3|\mathcal{F}_{m-1}]\varphi''(U)), \\
A_4 & =  \mathbf{E}[R_4(U,V_x)]+ \mathbf{E}[R_4(U,V_z)].
\end{align*}
We note that the two terms of $A_4$ have identical estimates. We further define
\begin{align*}
\beta_{r,m}^{(p)}&=\|\mathbf{E}[X_m^r|\mathcal{F}_{m-1}] \|_{p} \\
\gamma_{r,m}^{(p)}&=\|\mathbf{E}[X_m^r|\mathcal{F}_{m-1}] - \mathbf{E}[X_m^r]\|_{p}.
\end{align*}
 
The next step is to estimate each term in \eqref{a1toa4} to bound the right hand side of \eqref{bounddelta} by a constant independent of $n$ times $n^{-1/2}.$ We note an implication of H\"older's inequality for the sums above, which we will use for the estimates below. 
\begin{equation} \label{holder}
 |\mathbf{E}((\mathbf{E}[X_m^2|\mathcal{F}_{m-1}]-\sigma_m^2)\varphi'(u))|\leq \|\mathbf{E}[X_m^2|\mathcal{F}_{m-1}]-\sigma_m^2 \|_p \|\varphi'(u)\|_q,
 \end{equation}
 where $\frac{1}{p}+\frac{1}{q}=1.$

 \subsubsection{$A_2$}Since the normal density $\varphi$ and its derivatives are bounded,
\begin{equation*}\label{2short}
 A_2 \leq C \, \lambda_m^{-2}s_n^{-2} \gamma_{2,m}^{(p)} \leq C \frac{\sigma_m^2}{\sqrt{m}(s_n^2-s_m^2+\kappa^2)},
 \end{equation*}
by \eqref{holder} and \eqref{bound2}. Summing the quantity above over $m$ up to a linear order of $n$ yields small enough bounds, of order less than or equal to $n^{-1/2}.$ Define 
\begin{equation}\label{pim}
\pi_m=\frac{\sigma_m^2}{s_n^2-s_m^2+\kappa^2}
\end{equation}
We want to show that
\begin{equation}\label{21}
 \sup_n \sqrt{n} \sum_{m=1}^{\alpha n} \frac{\pi_m}{\sqrt{m}} < \infty.
 \end{equation} 
where $\alpha \in (0,1)$. Let us take $\kappa^2=\gamma \frac{s_n^2}{n}$ for $\gamma>0.$ Then by \eqref{1tom},
\begin{equation*}
 \sqrt{n} \sum_{m=1}^{\alpha n} \frac{\pi_m}{\sqrt{m}} \leq \frac{C\sqrt{n}}{s_n^2} \sum_{m=1}^{\alpha n} \frac{\sigma_m^2}{\sqrt{m}}.
 \end{equation*} 
It follows from \eqref{maincond} that, as in \eqref{boundondelta}, $\sigma_m^2 \leq \frac{(C+\delta)s_m^2}{m}$ except for finitely many $m$ for $\delta > 0.$ Therefore,  
\begin{equation*}
  \frac{C\sqrt{n}}{s_n^2} \sum_{m=1}^{\alpha n} \frac{\sigma_m^2}{\sqrt{m}} \leq  \frac{c\sqrt{n}}{s_n^2} \sum_{m=1}^{\alpha n} \frac{s_m^2}{m\sqrt{m}} + \frac{c'\sqrt{n}}{s_n^2} \leq C'
\end{equation*}
since $s_n^2$ is order larger or equal to $n$ and $\left| \frac{s_m^2}{s_m^2}\right| < 1 $ for $m < n.$ Note that $C'$ depends on $\alpha$ and $\delta$ but not on $n$. Thus, \eqref{21} holds true. However, better estimates are needed for large $m.$ The idea in \cite{Bh} is to use Lemma \ref{boldlem3} on $f=(\varphi')^q$ to obtain the bound
\begin{equation}\label{2long}
A_2 \leq  C \lambda_m^{-2}s_n^{-2} \gamma_{2,m}^{(p)}  \left(\delta(m-1)^{1/q}+ \left(\lambda_m \frac{s_n}{s_{m-1}}\right)^{1/q}\right).
 \end{equation} 
Let us  define $$K_n=\sqrt{n} \, \delta(n) \textnormal{ and } K_{(n)} = \max_{1\leq i \leq n} \sqrt{i} \, \delta(i),$$ which will be crucial at the end of the paper. The first term in \eqref{2long} times $\sqrt{n}$ added up over $m$ for the higher indices is bounded as
 \begin{equation} \label{22}
 \sqrt{n} \sum_{m=\alpha n}^{n} \frac{\sigma_m^2 (\delta(m-1)\sqrt{m-1})^{1/q}}{\sqrt{m}\sqrt[2q]{m-1}(s_n^2-s_m^2+ \kappa^2)} \leq  
 C n^{-\frac{1}{2q}} K_{(n-1)}^{1/q} \sum_{m=\alpha n}^{n} \pi_m.
\end{equation}
For the second term, we have
\begin{equation} \label{23}
  \sqrt{n} \sum_{m=\alpha n}^{n} \frac{1}{\sqrt{m}}\left(\frac{\sigma_m^2}{s_{m-1}^2}\right)^{1/2q}\pi_m^{1-\frac{1}{2q}} \leq  
 C n^{-\frac{1}{2q}}  \sum_{m=\alpha n}^{n} \pi_m^{1-\frac{1}{2q}},
 \end{equation} 
which follows from \eqref{maincond}. Then by the definition of \eqref{pim}, 
\begin{equation*}
  \pi_m =\frac{\sigma_m^2}{\sum_{k=m+1}^n \sigma_{k}^2 + \kappa^2}
  \end{equation*}  
 By  \eqref{varratiorange} and considering that $\sigma_n^2 \leq (C + \epsilon)\frac{s_n^2}{n},$ 
\begin{equation}\label{1overk}
\pi_m^2 \leq \frac{C+\epsilon}{ (n-m) \alpha^{\epsilon}(C+\epsilon) + \gamma} \leq C' \frac{1}{n-m + \gamma}.
\end{equation}
Therefore, we have
\begin{equation*}
\sum_{m=\alpha n}^n \pi_m^s \leq C' \sum_{k=1}^{(1-\alpha) n} \left(\frac{1}{k+\gamma}\right)^s.
\end{equation*}
Observe that if $s=1$, which is the case excluded in Lemma \ref{lemma3}, then the sum is bounded by $C' \log n.$ This verifies \eqref{22} is uniformly bounded. Whereas the case with \eqref{23} is covered by Lemma \ref{lemma3}, which guarantees its uniform boundedness. 

\vspace*{2mm}

\subsubsection{$A_3$} $A_3$ is treated in a similar way to $A_2$. By \eqref{holder} and \eqref{bound3},
\begin{equation*}
A_3 \leq C  \, \lambda_m^{-3}s_n^{-3} \beta_{3,m}^{(p')} \leq C \frac{1}{\sqrt[2p']{m}}\pi_m^{3/2}.
\end{equation*}
We first want to show
\begin{equation}\label{31}
\sup_n \sqrt{n} \sum_{m=1}^{\alpha n} \frac{\pi_m^{3/2}}{\sqrt[2p']{m}}<\infty.
\end{equation}
Again, by \eqref{maincond} and \eqref{1tom}, we have 
\begin{equation*}
 \sqrt{n} \sum_{m=1}^{\alpha n} \frac{\pi_m}{\sqrt{m}} \leq \frac{C\sqrt{n}}{s_n^3} \sum_{m=1}^{\alpha n} \frac{\sigma_m^3}{\sqrt[2p']{m}}\leq \frac{c\sqrt{n}}{s_n^3} \sum_{m=1}^{\alpha n} \frac{s_m^3}{m^{\frac{3}{2}+\frac{1}{2p'}}} + \frac{c'\sqrt{n}}{s_n^3} \leq C',
 \end{equation*} 
which verifies \eqref{31}. For larger values of $m$, we have the estimate
\begin{equation*}
 A_3 \leq C \, \lambda_m^{-3}s_n^{-3} \beta_{3,m}^{(p')}  \left(\delta(m-1)^{1/q'}+ \left(\lambda_m \frac{s_n}{s_{m-1}}\right)^{1/q'}\right)
 \end{equation*}
where $\frac{1}{p'}+\frac{1}{q'}=1.$ Multiplying by $\sqrt{n}$ and adding the terms up, we have
\begin{align}
 &C\sqrt{n} \sum_{m=\alpha n}^{n} \frac{1}{\sqrt[2p']{m}}\pi_m^{3/2}\frac{(\delta(m-1)\sqrt{m-1})^{\frac{1}{q'}}}{\sqrt[2q']{m}} \notag \\ 
\leq &C\sqrt{n} K_{(n-1)}^{\frac{1}{q'}}  \sum_{m=\alpha n}^{n}  \frac{1}{\sqrt{m}} \pi_m^{3/2} \leq C K_{(n-1)}^{\frac{1}{q'}} \sum_{m=\alpha n}^{n} \pi_m^{3/2}. \label{32} 
\end{align}
For the second term,
\begin{align}
 & C \sqrt{n} \sum_{m=\alpha n}^{n} \frac{1}{\sqrt[2p']{m}}\pi_m^{3/2}\left(\frac{s_n^2-s_m^2+ \kappa^2}{s_{m-1}^2}\right)^{\frac{1}{2q'}} \notag\\ 
\leq & C \sqrt{n} \sum_{m=\alpha n}^{n} \frac{1}{\sqrt[2p']{m}} \pi_m^{\frac{3}{2}-\frac{1}{2q'}} \frac{1}{\sqrt[2q']{m}}  \notag\\
\leq& C \sum_{m=\alpha n}^{n} \pi_m^{\frac{3}{2}-\frac{1}{2q'}}, \label{33}
\end{align} 
where the second inequality follows from \eqref{maincond}. Then by \eqref{1overk} and Lemma \ref{lemma3}, both \eqref{32} and \eqref{33} are uniformly bounded.

\subsubsection{$A_4$}The final terms are the remainders of the Taylor expansion, for which we take $R_4(u,v)=\frac{v^4}{24}\varphi^{(3)}(u- \theta v)$ for some $0 \leq \theta \leq 1.$ For  $p=\infty$ and $q=1,$ we have
\begin{equation*}
A_4 \leq C  \, \lambda_m^{-4}s_n^{-4} \beta_{4,m}^{\infty} \leq C\pi_m^{2}
\end{equation*}
by \eqref{bound4}. Similar to the first two cases, we can show that
\begin{equation*}
\sup_n \sqrt{n} \sum_{m=1}^{\alpha n} \pi_m^2  \leq \frac{C\sqrt{n}}{s_n^4} \sum_{m=1}^{\alpha n} \sigma_m^4 \leq \frac{c\sqrt{n}}{s_n^3} \sum_{m=1}^{\alpha n} \frac{s_m^4}{m^2} + \frac{c'\sqrt{n}}{s_n^4} \leq C'.
\end{equation*}
For the part of the sum with higher indices ($\alpha n \leq m \leq n$), further work is required. Let $\mathbf{1}_A$ denote the indicator function of the event $A.$ Consider the four regions below:\\
$\Gamma_1:=\{|U|<1\}$.  In this case, 
\begin{align*}
\mathbf{E}\left[\frac{V_x^4}{24}\varphi^{(3)}(U-\theta V_x) \mathbf{1}_{\Gamma_1}\right]&\leq \frac{\beta_{4,m}^{\infty} \| \varphi^{(3)}\|_{\infty}}{24 \lambda_m^4 s_n^4} \mathbf{E}( \mathbf{1}_{|U| <1})\\
&\leq C \lambda_m^{-4}s_n^{-4} \beta_{4,m}^{\infty}  \left(\delta(m-1)+ \left(\lambda_m \frac{s_n}{s_{m-1}}\right)\right)
\end{align*}
where we bound $\mathbf{1}_{|u| <1}$ by a smoothly decreasing function and apply Lemma \ref{boldlem3}. \\
$\Gamma_2:=\left\{V_x \leq \frac{|U|}{2}, |U|\geq 1\right\}$. Observe that $|U-\theta V_x| \geq  \frac{|U|}{2}.$ Defining $\psi(u)=\sup_{|x|\geq |u|/2} |\varphi^{(3)}(x)|,$ we can apply Lemma \ref{boldlem3} to $\psi$ obtain
\begin{align*}
\mathbf{E}\left[\frac{V_x^4}{24}\varphi^{(3)}(U-\theta V_x) \mathbf{1}_{\Gamma_2}\right]
&\leq \mathbf{E}\left[\frac{V_x^4}{24}\psi(U) \mathbf{1}_{\Gamma_2}\right] \\
&\leq C \lambda_m^{-4}s_n^{-4} \beta_{4,m}^{\infty}  \left(\delta(m-1)+ \left(\lambda_m \frac{s_n}{s_{m-1}}\right)\right).
\end{align*}
$\Gamma_3:=\{\frac{|U|}{2}\leq V_x \leq U^2, |U| \geq 1\}$. In this region, $|R_4(U,V_x)|$ is uniformly bounded as below.
\begin{align*}
|R_4(U,V_x)\mathbf{1}_{\Gamma_3}| &\leq \sup_u\left\{\left|2\Phi(u)\right| +\left|u^2\varphi(u)\right|+\left|\frac{u^4}{2} \varphi'(u)\right|+ \left|\frac{u^6}{6} \varphi''(u)\right|\right\} \mathbf{P}(\Gamma_3)\\
& \leq C \lambda_m^{-4} s_n^{-4} \beta_{4,m}^{\infty}  \mathbf{E}(|U|^{-4} \mathbf{1}_{|U|\geq 1}) \\
& \leq C \lambda_m^{-4}s_n^{-4} \beta_{4,m}^{\infty}  \left(\delta(m-1)+ \left(\lambda_m \frac{s_n}{s_{m-1}}\right)\right),
\end{align*}
where the second line follows from Markov's inequality. \\
\noindent $\Gamma_4:=\{ V_x \geq U^2 \geq 1\}.$ Here we use the bound $|R_4(u,v_x)| \leq c |v|^3.$ See \cite{Re} for an estimate on this and some other constants used in this section. We then have,
\begin{align*}
\mathbf{E}\left[c |V_x|^3 \mathbf{1}_{\Gamma_4} \right] &\leq \lambda_m^{-3}s_n^{-3} \beta_{3,m}^{\infty}
\mathbf{P}(\Gamma_4) \\
& \leq  C \lambda_m^{-4}s_n^{-4} \beta_{3,m}^{\infty} \mathbf{E}[\left |U|^{-2} \mathbf{1}_{|U|\geq 1} \right] \\
&  \leq C \lambda_m^{-4}s_n^{-4} \beta_{3,m}^{\infty}  \left(\delta(m-1)+ \left(\lambda_m \frac{s_n}{s_{m-1}}\right)\right).
\end{align*}
where the first line is by H\"older's inequality for $p=\infty, q=1$ recalling that $V_x=\frac{X_m}{\lambda s_n}.$ Since $\beta_{3,m}^{\infty} \leq \beta_{4,m}^{\infty},$ the two derivations above can be put together to have
\begin{equation*}
A_4 \leq  C \pi_m^{2} \left(\delta(m-1)+ \left(\lambda_m \frac{s_n}{s_{m-1}}\right)\right).
 \end{equation*}
Multiplying by $\sqrt{n}$ and adding them up gives an expression that can be bounded by
\begin{equation}\label{43}
C K_{(n-1)} \sum_{m=\alpha n}^{n} \pi_m^{2} + C' \sum_{m=\alpha n}^{ n} \pi_m^{3/2}.
\end{equation}
We apply Lemma \ref{lemma3} to the second sum of \eqref{43}, while the first sum serves as the pivot term for the recursive bound in the final part of the proof. 

\subsection{Recursive bound}

Combining our estimates in \eqref{bounddelta} and recalling that $\kappa^2=\gamma \frac{s_n^2}{n}$ for $\gamma>0$, we have
\begin{equation*}
K_n \leq  C K_{(n-1)} \sum_{m=\alpha n}^{n} \pi_m^{2} +c.
\end{equation*}
We observe that $\sum_{m=\alpha n}^n \pi_m^2$ can be taken as small as desired by choosing $\gamma$ large enough. Hence,
\begin{equation*}
K_n \leq \beta  K_{(n-1)} +c.
\end{equation*}
for some $0< \beta <1$ and for some constant $c.$ This implies that  $K_n \leq K_{(n-1)}$ if $K_{(n-1)}\geq \frac{c}{1-\beta}.$ Thus, $K_n \leq \max \left\{\frac{c}{1-\beta}, K_{(n-1)} \right\}.$ So the conclusion is $\sup_n K_n < \infty,$ which proves the theorem. 
\bbox 
 
\section*{Acknowledgement}
The author would like to thank Jason Fulman for suggesting the problem and possible directions to follow.

\bibliographystyle{alpha}
\bibliography{Kiraz}

\begin{thebibliography}{CKSS72}

\bibitem[Aal77]{Aa}
O.~O. Aalen.
\newblock Weak convergence of stochastic integrals related to counting
  processes.
\newblock {\em Zeitschrift f\"{u}r Wahrscheinlichkeitstheorie und Verwandte
  Gebiete}, 38(4):261--277, 1977.

\bibitem[B\'09]{BoS}
M.~B\'ona.
\newblock Real zeros and normal distribution for statistics on stirling
  permutations defined by gessel and stanley.
\newblock {\em SIAM Journal on Discrete Mathematics}, 23(1):401--406, 2009.

\bibitem[B\'16]{Bo}
M.~B\'ona.
\newblock {\em Combinatorics of permutations}.
\newblock Chapman and Hall/CRC, 2016.

\bibitem[Ben73]{B}
E.~A. Bender.
\newblock Central and local limit theorems applied to asymptotic enumeration.
\newblock {\em Journal of Combinatorial Theory Series A}, 15(1):91--111, 1973.

\bibitem[Ber44]{Ber}
H.~Bergstr\"om.
\newblock On the central limit theorem.
\newblock {\em Scandinavian Actuarial Journal}, 1944(3-4):139--153, 1944.

\bibitem[Bol82]{Bh}
E.~Bolthausen.
\newblock Exact convergence rates in some martingale central limit theorems.
\newblock {\em The Annals of Probability}, 10:672--688, 3 1982.

\bibitem[BR89]{BR89}
Pierre Baldi and Yosef Rinott.
\newblock On normal approximations of distributions in terms of dependency
  graphs.
\newblock {\em The Annals of Probability}, pages 1646--1650, 1989.

\bibitem[CD17]{CD}
S.~Chatterjee and P.~Diaconis.
\newblock A central limit theorem for a new statistic on permutations.
\newblock {\em Indian Journal of Pure and Applied Mathematics}, 48(4):561--573,
  2017.

\bibitem[CKSS72]{CK}
L.~Carlitz, D.~C. Kurtz, R.~Scoville, and O.~P. Stackelberg.
\newblock Asymptotic properties of {Eulerian} numbers.
\newblock {\em Zeitschrift f\"{u}r Wahrscheinlichkeitstheorie und Verwandte
  Gebiete}, 23(1):47--54, 1972.

\bibitem[CRS66]{Car}
L.~Carlitz, D.~P. Roselle, and R.~A. Scoville.
\newblock Permutations and sequences with repetitions by number of increases.
\newblock {\em Journal of Combinatorial Theory}, 1(3):350--374, 1966.

\bibitem[DB62]{DB}
F.~David and D.~Barton.
\newblock {\em Combinatorial Chance}.
\newblock Hafner Publishing Co., 1962.

\bibitem[DF85]{DF}
J.~D\'esarm\'enien and D.~Foata.
\newblock Fonctions sym\'etriques et s\'eries hyperg\'eom\'etriques basiques
  multivari\'ees.
\newblock {\em Bulletin de la Soci\'et\'e Math\'ematique de France}, 113:3--22,
  1985.

\bibitem[DMP95]{DMP}
P.~Diaconis, M.~McGrath, and J.~Pitman.
\newblock Riffle shuffles, cycles, and descents.
\newblock {\em Combinatorica}, 15(1):11--29, 1995.

\bibitem[Fel68]{Fe}
W.~Feller.
\newblock {\em An introduction to probability theory and its applications
  Vol.1}.
\newblock Wiley, New York, NY, 1968.

\bibitem[Foa10]{Fo}
D.~Foata.
\newblock Eulerian polynomials: from {Euler's} time to the present.
\newblock In {\em The legacy of Alladi Ramakrishnan in the mathematical
  sciences}, pages 253--273, New York, 2010. Springer.

\bibitem[Fro10]{Fr}
G.~Frobenius.
\newblock {Uber die Bernoullischen und die Eulerschen Polynome}.
\newblock {\em Sitzungsberichte der Preussische Akademie der Wissenschaften},
  pages 809--847, 1910.

\bibitem[FS09]{FS}
P.~Flajolet and R.~Sedgewick.
\newblock {\em Analytic Combinatorics}.
\newblock Cambridge University Press, 2009.

\bibitem[Ful98]{Fu}
J.~Fulman.
\newblock The distribution of descents in fixed conjugacy classes of the
  symmetric groups.
\newblock {\em Journal of Combinatorial Theory Series A}, 84(2):171--180, 1998.

\bibitem[Ful04]{F}
J.~Fulman.
\newblock Stein's method and non-reversible {M}arkov chains.
\newblock {\em Institute of Mathematical Statistics Lecture Notes-Monograph
  Series}, 46:66--74, 2004.

\bibitem[Ful06a]{F3}
J.~Fulman.
\newblock An inductive proof of the {Berry-Esseen} theorem for character
  ratios.
\newblock {\em Annals of Combinatorics}, 10(3):319--332, 2006.

\bibitem[Ful06b]{F2}
J.~Fulman.
\newblock Martingales and character ratios.
\newblock {\em Transactions of the American Mathematical Society},
  358(10):4533--4552, 2006.

\bibitem[Gon44]{G}
V.~Goncharov.
\newblock Du domaine d'analyse combinatoire.
\newblock {\em Izv. Akad. Nauk SSSR Ser. Mat}, 8:3--48, 1944.

\bibitem[GS78]{GS}
I.~Gessel and R.~P. Stanley.
\newblock Stirling polynomials.
\newblock {\em Journal of Combinatorial Theory Series A}, 24(1):24--33, 1978.

\bibitem[GZ06]{GZ}
V.~J.~W. Guo and J.~Zeng.
\newblock The {Eulerian} distribution on involutions is indeed unimodal.
\newblock {\em Journal of Combinatorial Theory Series A}, 113(6):1061--1071,
  2006.

\bibitem[Hae88]{Hae}
E.~Haeusler.
\newblock On the rate of convergence in the central limit theorem for
  martingales with discrete and continuous time.
\newblock {\em The Annals of Probability}, pages 275--299, 1988.

\bibitem[Har67]{Ha}
L.~H. Harper.
\newblock Stirling behavior is asymptotically normal.
\newblock {\em The Annals of Mathematical Statistics}, 38(2):410--414, 1967.

\bibitem[Hel82]{He}
I.~S. Helland.
\newblock Central limit theorems for martingales with discrete or continuous
  time.
\newblock {\em Scandinavian Journal of Statistics}, 9(2):79--94, 1982.

\bibitem[HH14]{HH}
P.~Hall and C.~Heyde.
\newblock {\em Martingale limit theory and its application}.
\newblock Academic Press, New York, 2014.

\bibitem[HR09]{Ho}
C.~Houdr\'e and R.~Restrepo.
\newblock A probabilistic approach to the asymptotics of the length of the
  longest alternating subsequence.
\newblock {\em The Electronic Journal of Combinatorics}, 16(\#R00), 2009.

\bibitem[HV12]{HV}
J.~Haglund and M.~Visontai.
\newblock Stable multivariate {Eulerian} polynomials and generalized {Stirling}
  permutations.
\newblock {\em European Journal of Combinatorics}, 33(4):477--487, 2012.

\bibitem[Kim17]{Ki}
G.~B. Kim.
\newblock {\em Distribution of descents in matchings}.
\newblock PhD thesis, University of Southern California, 2017.

\bibitem[Kim19]{K}
G.~B. Kim.
\newblock Distribution of descents in matchings.
\newblock {\em Annals of Combinatorics}, 23(1):73--87, 2019.

\bibitem[Kou94]{Kt}
M.~V. Koutras.
\newblock Eulerian numbers associated with sequences of polynomials.
\newblock {\em Fibonacci Quart}, 32(1):44--57, 1994.

\bibitem[Lin22]{L}
J.~W. Lindeberg.
\newblock {Eine neue Herleitung des Exponentialgesetzes in der
  Wahrscheinlichkeitsrechnung}.
\newblock {\em Mathematische Zeitschrift}, 15(1):211--225, 1922.

\bibitem[LW07]{LW}
L.~L. Liu and Y.~Wang.
\newblock A unified approach to polynomial sequences with only real zeros.
\newblock {\em Advances in Applied Mathematics}, 38(4):542--560, 2007.

\bibitem[Pet75]{Pet}
V.~V. Petrov.
\newblock {\em Sums of independent random variables}, volume~82.
\newblock Springer, Berlin, 1975.

\bibitem[Pet13]{Pe}
T.~K. Petersen.
\newblock Two-sided {Eulerian} numbers via balls in boxes.
\newblock {\em Mathematics Magazine}, 86(3):159--176, 2013.

\bibitem[Pet15]{PEN}
T.~K. Petersen.
\newblock {\em Eulerian numbers}.
\newblock Birkh\"{a}user, New York, 2015.

\bibitem[Pit97]{Pit}
J.~Pitman.
\newblock Probabilistic bounds on the coefficients of polynomials with only
  real zeros.
\newblock {\em Journal of Combinatorial Theory Series A}, 77(2):279--303, 1997.

\bibitem[Ren96]{Re}
J.~Renz.
\newblock A note on exact convergence rates in some martingale central limit
  theorems.
\newblock {\em The Annals of Probability}, 24(3):1616--1637, 1996.

\bibitem[Rio68]{R}
J.~Riordan.
\newblock {\em Combinatorial identities}, volume~6.
\newblock Wiley, New York, 1968.

\bibitem[RR99]{RR}
Y.~Rinott and V.~Rotar.
\newblock Some bounds on the rate of convergence in the {CLT} for martingales.
  {I}.
\newblock {\em Theory of Probability \& Its Applications}, 43(4):604--619,
  1999.

\bibitem[SS06]{SQ}
Q.~Shao and Z.~Su.
\newblock The {Berry-Esseen} bound for character ratios.
\newblock {\em Proceedings of the American Mathematical Society},
  134:2153--2159, 2006.

\bibitem[Sta86]{EC1}
R.~P. Stanley.
\newblock {\em Enumerative Combinatorics. Vol. I}.
\newblock The Wadsworth Brooks/Cole Mathematics Series, 1986.

\bibitem[Sta89]{Su}
R.~P. Stanley.
\newblock Log-concave and unimodal sequences in algebra, combinatorics, and
  geometry.
\newblock {\em Annals of the New York Academy of Sciences}, 576(1):500--535,
  1989.

\bibitem[Sta08]{St}
R.~P. Stanley.
\newblock Longest alternating subsequences of permutations.
\newblock {\em Michigan Mathematical Journal}, 57(1):675--687, 2008.

\bibitem[SV15]{SV}
C.~Savage and M.~Visontai.
\newblock The s-{Eulerian} polynomials have only real roots.
\newblock {\em Transactions of the American Mathematical Society},
  367(2):1441--1466, 2015.

\bibitem[Tan73]{T}
S.~Tanny.
\newblock A probabilistic interpretation of {Eulerian} numbers.
\newblock {\em Duke Mathematical Journal}, 40(4):717--722, 1973.

\bibitem[Wid06]{W}
H.~Widom.
\newblock On the limiting distribution for the length of the longest
  alternating sequence in a random permutation.
\newblock {\em Electronic Journal of Combinatorics}, 13(R25), 2006.

\bibitem[WS96]{WS}
D.~I. Warren and E.~Seneta.
\newblock Peaks and {Eulerian} numbers in a random sequence.
\newblock {\em Journal of Applied Probability}, 33(1):101--114, 1996.

\end{thebibliography}

\end{document}